\documentclass[11pt,reqno]{amsart}

\usepackage{tikz}

\usepackage[T1]{fontenc}
\usepackage{amsmath}							
\usepackage{amssymb}
\usepackage{amsthm}
\usepackage{amscd}
\usepackage{amsfonts}
\usepackage{mathabx}
\usepackage{stmaryrd}
\usepackage[all]{xy}

\usepackage{caption}
\usepackage{subcaption}
\usepackage{euler}

\usepackage{extarrows}

\usepackage[colorlinks, linktocpage, citecolor = purple, linkcolor = blue]{hyperref}
\usepackage{color}

\usepackage{tikz}									
\usetikzlibrary{matrix}
\usetikzlibrary{patterns}
\usetikzlibrary{positioning}
\usetikzlibrary{decorations.pathmorphing}
\usetikzlibrary{cd}

\usepackage{ytableau}
\usepackage{fullpage}
\usepackage[shortlabels]{enumitem}

\newtheorem{theorem}{Theorem}[section]
\newtheorem{lemma}[theorem]{Lemma}
\newtheorem{proposition}[theorem]{Proposition}

\theoremstyle{definition}
								
\newtheorem{definition}[theorem]{Definition}
\newtheorem{example}[theorem]{Example}

\newtheorem{remark}[theorem]{Remark}
\newtheorem*{remark*}{Remark}

\newcommand{\ZZ}{\mathbb{Z}}
\newcommand{\XX}{\mathbb{X}}
\newcommand{\YY}{\mathbb{Y}}

\newcommand{\QQ}{\mathbb{Q}}

\renewcommand{\th}{\widetilde{h}}

\newcommand{\tv}{\widetilde{v}}

\newcommand{\tX}{\widetilde{X}}

\DeclareMathOperator{\Aut}{Aut}
\DeclareMathOperator{\Ker}{Ker}
\DeclareMathOperator{\Coker}{Coker}
\DeclareMathOperator{\Hom}{Hom}

\let\Im\relax
\DeclareMathOperator{\Im}{Im}

\newcommand{\calO}{\mathcal{O}}

\newcommand{\calX}{\mathcal{X}}

\DeclareMathOperator{\tr}{tr}
\DeclareMathOperator{\adj}{adj}

\DeclareMathOperator{\Prin}{Prin}

\DeclareMathOperator{\val}{val}

\DeclareMathOperator{\Id}{Id}

\DeclareMathOperator{\Jac}{Jac}

\newcommand{\dil}{\mathrm{dil}}
\newcommand{\fr}{\mathrm{fr}}
\newcommand{\ep}{\varepsilon}

\title{Chip-firing on graphs of groups}
\author{Margaret Meyer, Dmitry Zakharov}

\begin{document}

\begin{abstract} We define the Laplacian matrix and the Jacobian group of a finite graph of groups. We prove analogues of the matrix tree theorem and the class number formula for the order of the Jacobian of a graph of groups. Given a group $G$ acting on a graph $X$, we define natural pushforward and pullback maps between the Jacobian groups of $X$ and the quotient graph of groups $X/\!/G$. For the case $G=\ZZ/2\ZZ$, we also prove a combinatorial formula for the order of the kernel of the pushforward map.

\end{abstract}

\maketitle

\section{Introduction}

The theory of chip-firing on graphs is a purely combinatorial theory, having a remarkable similarity to divisor theory on algebraic curves. A divisor on a graph is an integer linear combination of its vertices, and two divisors are linearly equivalent if one is obtained from another by a sequence of chip-firing moves. The set of equivalence classes of degree zero divisors on a graph $X$ is a finite abelian group, called the \emph{Jacobian} $\Jac(X)$ or the \emph{critical group} of $X$. The similarity with algebraic geometry is not accidental: graphs record degeneration data of one-dimensional families of algebraic curves, and divisors on graphs represent discrete invariants of algebraic divisor classes under degeneration. 

Chip-firing on graphs is functorial with respect to a class of graph maps known as \emph{harmonic morphisms}, which may be viewed as discrete analogues of finite maps of algebraic curves. Specifically, a harmonic morphism of graphs $f:X\to Y$ defines natural pushforward and pullback maps $f_*:\Jac(X)\to \Jac(Y)$ and $f^*:\Jac(Y)\to \Jac(X)$. Harmonic morphisms are characterized by a local degree assignment at the vertices of the source graph, and are a generalization of topological coverings, which have local degree one everywhere. 

A natural example of a topological covering, and hence of a harmonic morphism, is the quotient $p:X\to X/G$ of a graph $X$ by a free action of a group $G$. The paper~\cite{2014ReinerTseng} thoroughly investigated the corresponding pushforward map $p_*:\Jac(X)\to \Jac(X/G)$, and found a combinatorial formula for the degree of the kernel in the case when $G=\ZZ/2\ZZ$. If the action of $G$ on $X$ has nontrivial stabilizers, however, then $p$ is not in general harmonic, and there is no relationship between $\Jac(X)$ and $\Jac(X/G)$. This raises the natural problem of redefining chip-firing on the quotient graph in a way that preserves functoriality.

In this paper, we solve this problem using the theory of \emph{graphs of groups}, also known as Bass--Serre theory (see~\cite{1993Bass} and~\cite{2002Serre}). Given a $G$-action on a graph $X$, the \emph{quotient graph of groups} $X/\!/G$ consists of the quotient graph $X/G$ together with the data of the local stabilizers, and may be thought of as the stacky quotient of $X$ by $G$. We define the Laplacian matrix and the Jacobian group of a graph of groups by weighting the chip-firing map using the orders of the local stabilizers. We define natural pushforward and pullback maps $p_*:\Jac(X)\to \Jac(X/\!/G)$ and $p^*:\Jac(X/\!/G)\to \Jac(X)$, and we investigate their properties. 

The paper is organized as follows. In Section 2, we recall the definitions of chip-firing for a graph, as well as harmonic morphisms of graphs and Bass--Serre theory. We define graphs and chip-firing in terms of \emph{half-edges} and introduce a detailed factorization of the graph Laplacian. This approach is notationally cumbersome but proves useful in Section 3, where we define chip-firing and the Jacobian group for a graph of groups. We prove two formulas for the order of the Jacobian of a graph of groups: Theorem~\ref{thm:matrixtree}, a weighted version of Kirchhoff's matrix tree theorem, and Theorem~\ref{thm:zeta}, which is a class number formula involving a hypothetical Ihara zeta function of a graph of groups. In Section 4, we consider a group $G$ acting on a graph $X$ and consider the Jacobian of the quotient graph of groups $X/\!/G$. We define natural pushforward and pullback maps between the Jacobians $\Jac(X)$ and $\Jac(X/\!/G)$. We compute the Jacobians of all group quotients of two graphs with large automorphism groups: the complete graph on four vertices and the Petersen graph. Finally, in Section 5 we specialize to the case $G=\ZZ/2\ZZ$ and find a combinatorial formula for the order of the kernel of the pushforward map $p_*:\Jac(X)\to \Jac(X/\!/G)$, generalizing a result of Reiner and Tseng~\cite{2014ReinerTseng}. 

A natural question is to relate chip-firing on graphs of groups to algebraic geometry. A version of the chip-firing maps with edge weights (but trivial vertex weights) appears in~\cite{2015Chiodo} and~\cite{2017ChiodoFarkas}, in the study of moduli spaces of curves with level structure. Curves with a $G$-cover with arbitrary group $G$ are considered in~\cite{2019Galeotti}. It is natural to assume that chip-firing on graphs of groups should be related to the theory of line bundles on stacky curves. Investigating this connection, however, is beyond the scope of this paper.

\section{Graphs with legs and graphs of groups}

We begin by recalling a number of standard definitions concerning graphs, group actions, divisor theory on graphs, harmonic morphisms, and graphs of groups.

\subsection{Graphs, morphisms, and group actions} In Serre's definition (see~\cite{2002Serre}), the edges of a graph are the orbits of a fixed-point-free involution acting on a set of \emph{half-edges}. When considering group actions on graphs, it is then necessary to require that the action not flip any edges of the graph. We can relax this constraint by allowing the involution on the set of half-edges to have fixed points. The resulting object is a \emph{graph with legs}, where a leg is the result of folding an edge in half via an involution. Such objects have appeared before in the combinatorics literature (for example, see p.~60 in the paper~\cite{1982Zaslavsky}, where they are called \emph{half-arcs}).

\begin{definition} A \emph{graph with legs} $X$, or simply a \emph{graph}, consists of the following data:
\begin{enumerate}
    \item A set of \emph{vertices} $V(X)$.
    \item A set of \emph{half-edges} $H(X)$.
    \item A \emph{root map} $r_X:H(X)\to V(X)$.
    \item An involution $\iota_X:H(X)\to H(X)$.
\end{enumerate}
The involution $\iota_X$ partitions $H(X)$ into orbits of size one and two. An orbit $e=\{h,h'\}$ of size two (so that $\iota_X(h)=h'$) is an \emph{edge} with \emph{root vertices} $r_X(h),r_X(h')\in V(X)$, and the set of edges of $X$ is denoted $E(X)$. An edge whose root vertices coincide is called a \emph{loop}. A fixed point of $\iota_X$ is called a \emph{leg} and has a single root vertex $r_X(h)\in V(X)$, and we denote the set of legs of $X$ by $L(X)$. The \emph{tangent space} $T_vX=r_X^{-1}(v)$ of a vertex $v\in V(X)$ is the set of half-edges rooted at $v$, and its \emph{valency} is $\val(v)=|T_vX|$ (so a leg is counted once, while a loop is counted twice). An \emph{orientation} of an edge $e=\{h,h'\}$ is a choice of order $(h,h')$ on the half-edges, and we call $s(e)=r_X(h)$ and $t(e)=r_X(h')$ respectively the \emph{initial} and \emph{terminal} vertices of an oriented edge $e$. An \emph{orientation} $\calO$ on $X$ is a choice of orientation for each edge (each leg has a unique orientation). We consider only finite connected graphs.

\end{definition}

\begin{definition} A \emph{morphism of graphs} $f:\tX\to X$ is a pair of maps $f:V(\tX)\to V(X)$ and $f:H(\tX)\to H(X)$ (both denoted $f$ by abuse of notation) that commute with the root and involution maps on $\tX$ and $X$.

\end{definition}

Let $f:\tX\to X$ be a morphism of graphs. If $l\in L(\tX)$ is a leg then $\iota_X(f(l))=f(\iota_{\tX}(l))=f(l)$, so $f(l)\in L(X)$ is also a leg. On the other hand, if $e=\{h,h'\}\in E(\tX)$ is an edge, then either $f(h)\neq f(h')$, in which case $f$ maps $e$ to an edge $f(e)=\{f(h),f(h')\}\in E(X)$, or $f(h)=f(h')\in L(X)$ is a leg. In other words, edges can map to edges or fold to legs. However, we do not allow morphisms to contract edges or half-legs, in other words we consider only \emph{finite} morphisms.

\begin{definition} Let $X$ be a graph and let $G$ be a group acting on the right on $X$. In other words, each $g\in G$ defines an automorphism of $X$, which we denote $x\mapsto xg$ for $x\in V(X)\cup H(X)$, such that $x(g_1g_2)=(xg_1)g_2$ for all $x\in V(X)\cup H(X)$ and all $g_1,g_2\in G$. We define the vertices and half-edges of the \emph{quotient graph} $X/G$ as the $G$-orbits of $V(X)$ and $H(X)$:
\[
V(X/G)=V(X)/G=\{vG:v\in V(X)\},\quad 
H(X/G)=H(X)/G=\{hG:h\in H(X)\},
\]
and descending the root and involution maps:
\[
r_{X/G}(hG)=r_X(h)G,\quad \iota_{X/G}(hG)=\iota_X(h)G.
\]
The quotient projection $p:X\to X/G$ sends each element of $X$ to its orbit.

Let $h\in H(X)$ be a half-edge with orbit $p(h)=hG\in H(X/G)$. If $h$ is a leg, then $\iota_{X/G}(hG)=\iota_{X}(h)G=hG$ so $p(h)=hG\in L(X/G)$ is also a leg. If $h$ belongs to an edge $e=\{h,h'\}\in E(X)$, then there are two possibilities. If $h'\neq hg$ for all $g\in G$, then the orbits $hG$ and $h'G$ are distinct half-edges of $X/G$ forming an edge $p(e)=\{hG,h'G\}\in E(X/G)$. However, if $h'=hg$ for some $g\in G$ (in other words, if the $G$-action \emph{flips the edge $e$}), then $p(e)=hG=h'G\in L(X/G)$ is a leg.

\end{definition}

In Serre's original definition, the involution $\iota_X$ on a graph $X$ is required to be fixed-point-free, and hence the set $H(X)$ of half-edges is partitioned into edges only. Relaxing this condition enables us to consider quotients by group actions that flip edges. We give a simple example below and two extended examples in Sections~\ref{subsec:tetrahedron} and~\ref{subsec:petersen}.

\begin{example} Let $X$ be the graph with two vertices joined by an edge. There is a unique nontrivial morphism $f:X\to X$ exchanging the two vertices, so $\Aut(X)$ is the cyclic group of order two. The quotient $X/\Aut(X)$ is the graph having one leg at one vertex, and is in fact the terminal object in the category of graphs with legs, while no such object exists in the category of graphs. 

\end{example}


\subsection{The graph Laplacian and chip-firing.} We now recall divisor theory on a graph $X$. We follow the framework of the paper~\cite{2014ReinerTseng}, which we reformulate in terms of half-edges. Specifically, we use a detailed factorization of the Laplacian which can be conveniently generalized to graphs of groups. A minor additional advantage is that we are never required to pick an orientation for the graph.

For a set $S$, we denote $\ZZ^S$ and $\ZZ^S_0$ respectively the free abelian group on $S$ and the subgroup consisting of elements whose coefficients sum to zero. The free abelian group $\ZZ^{V(X)}$ is called the \emph{divisor group} of $X$, and a \emph{divisor} $D=\displaystyle\sum_{v\in V(X)} a_v v$ is interpreted as a distribution of $a_v$ chips on each vertex $v$. The root and involution maps $r_X:H(X)\to V(X)$ and $\iota_X:H(X)\to H(X)$ induce homomorphisms
\[
r_X:\ZZ^{H(X)}\to \ZZ^{V(X)},\quad\iota_X:\ZZ^{H(X)}\to \ZZ^{H(X)}
\]
on the corresponding free abelian groups (denoted by the same letters by abuse of notation). Let $\tau_X$ denote the transpose of $r_X$:
\begin{equation}
\tau_X:\ZZ^{V(X)}\to \ZZ^{H(X)},\quad \tau_X(v)=\sum_{h\in T_vX}h.
\label{eq:tau}
\end{equation}

\begin{definition} The \emph{Laplacian} of a graph $X$ is the homomorphism $L_X:\ZZ^{V(X)}\to \ZZ^{V(X)}$ given by
\begin{equation}
L_X=r_X\circ(Id-\iota_X)\circ \tau_X,\quad L_X(v)=\sum_{h\in T_vX}(v-r_X(\iota_X(h))).
\label{eq:LX}
\end{equation}
\end{definition}

\begin{figure}
    \centering
\begin{tikzcd}
\ZZ^{H(X)}\arrow[r,bend left,"r_X"] \arrow[loop left,"\iota_X"]& \ZZ^{V(X)} \arrow[l,bend left,"\tau_X"'] \arrow[loop right,"L_X"]
\end{tikzcd}
    \caption{Factorization of the graph Laplacian.}
    \label{fig:1}
\end{figure}
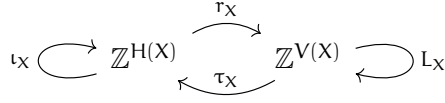



Figure~\ref{fig:1} displays all the maps involved in defining the graph Laplacian. It is elementary to verify that $\Im L_X\subset \ZZ^{V(X)}_0$, where $\Im L_X$ is the subgroup of \emph{principal divisors} on $X$, and in fact $\ZZ^{V(X)}_0=\Im (r_X\circ(Id-\iota_X))$ if the graph $X$ is connected.

\begin{definition} The \emph{Jacobian} of a graph $X$ is the quotient group
\[
\Jac(X)=\ZZ^{V(X)}_0/\Im L_X.=\Im(r_X\circ (\Id-\iota_X))/\Im L_X.
\]
\end{definition}

The Jacobian $\Jac(X)$ is also known as the \emph{critical group} of $X$. Kirchhoff's matrix-tree theorem states that $\Jac(X)$ is a finite group whose order is equal to the number of spanning trees of $X$. Given a vertex $v\in V(X)$, the divisor $-L_X(v)$ is obtained by \emph{firing the vertex} $v$, in other words by moving a chip from $v$ along each half-edge $h\in T_vX$ to the root vertex of $\iota_X(h)$. Chips moved along legs and loops return to $v$, hence legs and loops of $X$ do not contribute to the Laplacian or the Jacobian group, and $\Jac(X)$ is canonically isomorphic to the Jacobian of the graph obtained by removing all legs and loops. However, legs and loops naturally occur when taking quotients by group actions, so we nevertheless consider them.

We give an explicit presentation for the matrix $L$ of the graph Laplacian $L_X$. Let $n=|V(X)|$ and $m=|E(X)|$ denote the number of vertices and edges, respectively. Then $L=Q-A$, where $Q$ and $A$ are the $n\times n$ \emph{valency} and \emph{adjacency matrices} of $X$:
\[
L_{uv}=Q_{uv}-A_{uv},\quad
Q_{uv}=\delta_{uv}\val(v),\quad A_{uv}=|\{h\in T_vX:r_X(\iota_X(h))=u\}|.
\]
These matrices have the following convenient factorizations. Pick an orientation on $X$ and define the $n\times m$ \emph{root matrices}
\begin{equation}
    S_{ve}=\left\{\begin{array}{cc} 1, & s(e)=v, \\ 0, & s(e)\neq v, \end{array}\right.,\quad 
T_{ve}=\left\{\begin{array}{cc} 1, & t(e)=v, \\ 0, & t(e)\neq v. \end{array}\right.
\label{eq:adjacency}
\end{equation}
It is then easy to verify that 
\[
Q=SS^t+TT^t,\quad A=ST^t+TS^t,\quad L=Q-A=(S-T)(S-T)^t.
\]

\subsection{Harmonic morphisms of graphs} Given a morphism of graphs $f:\tX\to X$, there is generally no relationship between $\Jac(\tX)$ and $\Jac(X)$. However, we can define functoriality with respect to a class of graph morphisms that admit a local degree function on the vertices of the source graph (see~\cite{2000Urakawa} and ~\cite{2009BakerNorine}). 

\begin{definition} A graph morphism $f:\tX\to X$ is called \emph{harmonic} if there exists a function $d_f:V(\tX)\to \ZZ$, called the \emph{local degree}, such that for any $\tv\in V(\tX)$ and any $h\in T_{f(\tv)}X$ we have
\[
d_f(\tv)=\left|\left\{\th\in T_{\tv}\tX:f(\th)=h\right\}\right|.
\]
\end{definition}

For example, a covering space $f:\tX\to X$ (in the topological sense) is the same thing as a harmonic morphism with $d_f(\tv)=1$ for all $\tv\in V(\tX)$. If $X$ is connected, then any harmonic morphism $f:\tX\to X$ has a \emph{global degree} equal to
\[
\deg(f)=\sum_{\tv\in f^{-1}(v)}d_f(\tv)=|f^{-1}(h)|
\]
for any $v\in V(X)$ or any $h\in H(X)$. In particular, any harmonic morphism to a connected graph is surjective (on the edges and the vertices).

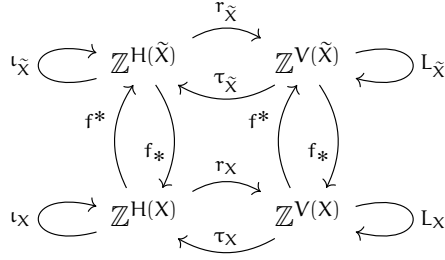
\begin{figure}
    \centering
\begin{tikzcd}
\ZZ^{H(\tX)}\arrow[r,bend left,"r_{\tX}"] \arrow[dd,bend left,"f_*"'] \arrow[loop left,"\iota_{\tX}"]& \ZZ^{V(\tX)} \arrow[l,bend left,"\tau_{\tX}"'] \arrow[dd,bend left,"f_*"']\arrow[loop right,"L_{\tX}"] \\
 & \\
\ZZ^{H(X)}\arrow[r,bend left,"r_X"] \arrow[uu,bend left,"f^*"] \arrow[loop left,"\iota_X"]& \ZZ^{V(X)} \arrow[l,bend left,"\tau_X"'] \arrow[uu,bend left,"f^*"] \arrow[loop right,"L_X"]
\end{tikzcd}
    \caption{Pushforward and pullback maps associated to a harmonic morphism.}
    \label{fig:2}
\end{figure}

Let $f:\tX\to X$ be a harmonic morphism of graphs, and denote 
\[
f_*:\ZZ^{V(\tX)}\to \ZZ^{V(X)},\quad f_*(\tv)=f(\tv),\quad
f_*:\ZZ^{H(\tX)}\to \ZZ^{H(X)},\quad f_*(\th)=f(\th)
\]
the induced homomorphisms on the free abelian groups. For any graph morphism (not necessarily harmonic) we have
\[
f_*\circ r_{\tX}=r_X\circ f_*,\quad f_*\circ \iota_{\tX}=\iota_X\circ f_*.
\]
For any $\tv\in V(\tX)$ we have 
\begin{equation}
(f_*\circ \tau_{\tX})(\tv)=d_f(\tv) (\tau_X\circ f_*)(\tv)
\label{eq:taulocaldegree1}
\end{equation}
by the harmonicity of $f$, therefore
\[
(f_*\circ L_{\tX})(\tv)=d_f(\tv)(L_X\circ f_*)(\tv).
\]
It follows that $f_*(\Im L_{\tX})\subset \Im L_X$ and the map $f_*$ descends to a surjective \emph{pushforward map}
\[
f_*:\Jac(\tX)\to \Jac(X).
\]

Similarly, if $X$ is connected, we define the maps 
\begin{equation}
f^*:\ZZ^{V(X)}\to \ZZ^{V(\tX)},\quad 
f^*(v)=\sum_{\tv\in f^{-1}(v)}d_f(\tv)\cdot\tv
\label{eq:pullbacklocaldegree}
\end{equation}
and 
\[
f^*:\ZZ^{H(X)}\to \ZZ^{H(\tX)},\quad 
f^*(h)=\sum_{\th\in f^{-1}(h)} \th.
\]
It is easy to verify that  
\[
f^*(L_X(v))=\sum_{\tv\in f^{-1}(v)}L_{\tX}(\tv)
\]
for any $v\in V(X)$, hence $f^*(\Prin(X))\subset \Prin (\tX)$ and there is an induced \emph{pullback map}
\[
f^*:\Jac(X)\to \Jac(\tX).
\]
The map $f^*:\Jac(X)\to \Jac(\tX)$ is injective (Theorem 4.7 in~\cite{2009BakerNorine}), and the composition $f_*\circ f^*$ acts by multiplication by $\deg (f)$ on $\Jac(X)$. Figure~\ref{fig:2} displays all the maps associated to a harmonic morphism of graphs.

\subsection{Graphs of groups} \label{subsec:gog} We now recall graphs of groups, which are the natural category for taking quotients of graphs by non-free group actions. We modify the definitions in~\cite{1993Bass} to allow graphs with legs (and thus quotients by group actions that flip edges). 

\begin{definition} A \emph{graph of groups} $\XX=(X,\calX_v,\calX_h)$ consists of the following data:
\begin{itemize}
    \item A graph $X$ (possibly with legs).
    
    \item A group $\calX_v$ for each vertex $v\in V(X)$.

    \item A subgroup $\calX_h\subset \calX_{r_X(h)}$ for each half-edge $h\in H(X)$.

    \item An isomorphism $i_h:\calX_h\to \calX_{\iota_X(h)}$ for each edge $\{h,\iota_X(h)\}\in E(X)$, where we assume that $i_{\iota_X(h)}=i_{h}^{-1}$. 


\end{itemize}

\end{definition} 


Our definition differs slightly from the standard one~\cite{1993Bass}, where one assumes that the two groups $\calX_h$ and $\calX_{\iota_X(h)}$ corresponding to an edge are the same, and instead records monomorphisms $\calX_h\to \calX_{r(h)}$. The two approaches are equivalent in the case when there are no legs. We consider only finite graphs of groups, so that the underlying graph and all vertex groups are finite. 

We now define the quotient graph of groups by a right group action on a graph. The standard definition in~\cite{1993Bass} uses a trivialization with respect to a choice of spanning tree in the quotient graph and a lift of the tree to the source graph, and records the gluing data on the complementary edges (with respect to a choice of orientation). We find it more natural to instead trivialize the neighborhood of every vertex.

\begin{definition} Let $G$ be a group acting on the right on a graph $\tX$, let $X=\tX/G$ be the quotient graph, and let $p:\tX\to X$ be the quotient map. We define the \emph{quotient graph of groups} $\tX/\!/G=(X,\calX_v,\calX_h)$ on $X$ as follows:
\begin{enumerate}
    \item Choose a section $\widetilde{(\cdot)}:V(X)\to V(\tX)$ of the map $p:V(\tX)\to V(X)$. For each vertex $v\in V(X)$, $\calX_v=G_{\tv}=\{g\in G:\tv g=\tv\}$ is the stabilizer of the chosen preimage $\tv\in p^{-1}(v)$.

    \item Choose a section $\widetilde{(\cdot)}:H(X)\to H(\tX)$ of the map $p:H(\tX)\to H(X)$ with the property that $r_{\tX}(\th)=\widetilde{r_X(h)}$ for all $h\in H(X)$. For each half-edge $h\in H(X)$, $\calX_h=G_{\th}=\{g\in G:\th g=\th\}$ is the stabilizer the chosen preimage $\th\in p^{-1}(h)$. It is clear that $\calX_h=G_{\th}\subset G_{r_{\tX}(\th)}=\calX_{r_X(h)}$.

    
\end{enumerate}

\end{definition}

For $v\in V(X)$ and $g\in G$ we denote $\tv_g=\tv g$ (so that $\tv_1=\tv$); this identifies the fiber $p^{-1}(v)=\{\tv_g:g\in G\}$ with the set $\calX_v\backslash G$ of right cosets of $\calX_v$ in $G$. Similarly, given $h\in H(X)$ and $g\in G$ we denote $\th_g=\th g$ (so that $\th_1=\th$), so that $p^{-1}(h)=\{\th_g:g\in G\}$ is identified with $\calX_h\backslash G$. Hence
\begin{equation}
V(\tX)=\coprod_{v\in V(X)}\calX_v\backslash G,\quad H(\tX)=\coprod_{h\in H(X)}\calX_h\backslash G
\label{eq:tXset}
\end{equation}
as sets, and under this identification the root and projection maps and the $G$-action are given by 
\begin{equation}
p(\tv_g)=v,\quad p(\th_g)=h,\quad r_{\tX}(\th_g)=\widetilde{r_X(h)}_g,\quad \tv_gg'=\tv_{gg'},\quad \th_gg'=\th_{gg'}
\label{eq:tXmap1}
\end{equation}
for $v\in V(X)$, $h\in H(X)$, and $g,g'\in G$.

Finally, let $h\in H(X)$ be a half-edge. Applying the involution on $\tX$ to $\th$ gives a half-edge lying over $h'=\iota_X(h)$ (it may be that $h'=h$). Therefore there exists an element $\beta(h)\in G$, unique up to left multiplication by $\calX_{h'}$, such that $\iota_{\tX}(\th)=\widetilde{h'}_{\beta(h)}$. It follows that 
\begin{equation}
\iota_{\tX}(\th_g)=\widetilde{\iota_X(h)}_{\beta(h)g}
\label{eq:tXmap2}
\end{equation}
for all $h\in H(X)$ and $g\in G$. We observe that $\calX_{h'}=\beta(h)\calX_h\beta(h)^{-1}$. We can choose the $\beta(h)$ so that $\beta(h')=\beta(h)^{-1}$ for all $h$ (in general, they only satisfy $\beta(h')\beta(h)\in \calX_h$). The required isomorphism $i_h:\calX_h\to \calX_{\iota_X(h)}$ is then given by conjugation by $\beta(h)$.

We can run the construction in reverse and recover the morphism $p:\tX\to X$ together with the $G$-action on $\tX$ from the quotient graph of groups $X/\!/G$ together with the chosen elements $\beta(h)\in G$ (in keeping with graph-theoretic terminology, we may call the $\beta(h)$ a \emph{generalized $G$-voltage assignment} on $X/\!/G$). First of all, we assume that the vertex and half-edge groups are given not simply as abstract groups, but as subgroups of $G$. Hence we can define $\tX$ as a set by Equation~\eqref{eq:tXset}. The root and projection maps are given by Equation~\eqref{eq:tXmap1}, so that $\tX$ is trivialized in the neighborhood of each vertex. Finally, the involution map is given by Equation~\eqref{eq:tXmap2} and defines how the tangent spaces of the vertices are glued to each other. We note that for an edge $\{h,h'\}\in E(X)$ we may choose $\beta(h)\in G$ arbitrarily and then set $\beta(h')=\beta(h)^{-1}$, but for a leg $h\in L(X)$ the element $\beta(h)\in G$ must have order two (or be the identity), and furthermore must lie in the normalizer of $\calX_h$. The fiber $p^{-1}(h)$ over the leg $h$ consists of legs if $\beta(h)\in \calX_h$ (in which case we may as well have chosen $\beta(h)=1$) and edges if $\beta(h)\notin \calX_h$.

Two generalized $G$-voltage assignments on $X/\!/G$ are equivalent if they define isomorphic $G$-covers $\tX\to X$. The set of equivalence classes of voltage assignments may be constructed as the first \v{C}ech cohomology set of an appropriate constructible sheaf of non-abelian groups on $X$. This set was explicitly described in~\cite{2019LenUlirschZakharov} for an abelian group $G$ (in which case the set is also an abelian group), and the construction immediately generalizes to the non-abelian case. We leave the details to the interested reader.

\section{The Laplacian and the Jacobian group of a graph of groups}

Let $G$ be a finite group acting on a finite graph $\tX$, and let $p:\tX\to X=\tX/G$ be the quotient map. If the action of $G$ is free, then $p$ is a covering space and hence a harmonic morphism, and induces pushforward and pullback homomorphisms $p_*:\Jac(\tX)\to \Jac(X)$ and $p^*:\Jac(X)\to \Jac(\tX)$. However, for an arbitrary $G$-action there is no natural relationship between $\Jac(\tX)$ and $\Jac(X)$. The solution is to replace $X$ with the quotient graph of groups $\XX=\tX/\!/G$, and to define the chip-firing operation on $\tX/\!/G$ in a way that takes into account the orders of the local stabilizers. We now describe this construction.

\subsection{Chip-firing on a graph of groups} Let $\XX=(X,\calX_v,\calX_h)$ be a graph of groups, and let $\ZZ^{V(X)}$ and $\ZZ^{H(X)}$ be the free abelian groups on the vertices and half-edges of the underlying graph, respectively. As for graphs, we call $\ZZ^{V(X)}$ the \emph{divisor group} of $\XX$, and interpret divisors as distributions of chips on the vertices of the underlying graph $X$ (the chips are not weighted in any way). As before, the root and involution maps induce homomorphisms
\[
r_X:\ZZ^{H(X)}\to \ZZ^{V(X)},\quad\iota_X:\ZZ^{H(X)}\to \ZZ^{H(X)}.
\]
For $v\in V(X)$ denote $c(v)=|\calX_v|$ the order of the local group at $v$, and similarly for $h\in H(X)$ denote $c(h)=|\calX_h|$. Given an edge $e=\{h,h'\}\in E(X)$, we denote $c(e)=c(h)=c(h')$. For each half-edge $h\in H(X)$ rooted at $v=r_X(h)$, there is an inclusion $\calX_h\subset \calX_v$ of the local groups, hence $c(h)$ divides $c(v)$. We now define the weighted transpose of $r_X$ by the formula
\begin{equation}
\tau_{\XX}:\ZZ^{V(X)}\to \ZZ^{H(X)},\quad \tau_{\XX}(v)=
\sum_{h\in T_vX}\frac{c(v)}{c(h)}h.
\label{eq:tauXX}
\end{equation}

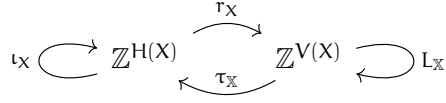
\begin{figure}
    \centering
\begin{tikzcd}
\ZZ^{H(X)}\arrow[r,bend left,"r_X"] \arrow[loop left,"\iota_X"]& \ZZ^{V(X)} \arrow[l,bend left,"\tau_{\XX}"'] \arrow[loop right,"L_{\XX}"]
\end{tikzcd}
    \caption{Factorization of the Laplacian of a graph of groups.}
    \label{fig:3}
\end{figure}

\begin{definition} The \emph{Laplacian} of the graph of groups $\XX=(X,\calX_v,\calX_h)$ is the homomorphism $L_{\XX}:\ZZ^{V(X)}\to \ZZ^{V(X)}$ given by
\begin{equation}
L_{\XX}=r_{X}\circ(Id-\iota_{X})\circ \tau_{\XX},\quad L_{\XX}(v)=\sum_{h\in T_vX}\frac{c(v)}{c(h)}(v-r_{X}(\iota_{X}(h))).
\label{eq:LXX}
\end{equation}

\end{definition}

Given a vertex $v\in V(G)$, the divisor $-L_{\XX}(v)$ is the result of \emph{firing} the vertex $v$. It is obtained by moving, along each edge $e=\{h,h'\}$ rooted at $v$, a stack of $c(v)/c(e)$ chips from $v$ to the other root vertex of $e$. As in the case of graphs, if $h$ is a leg or belongs to a loop then $r_X(h)=r_X(\iota_X(h))$, so loops and legs do not contribute to the Laplacian. However, the chip-firing operation is not symmetric: firing two adjacent vertices will in general cause them to exchange chips. As before, if $X$ is connected then $\ZZ^{V(X)}_0=\Im(r_{X}\circ(Id-\iota_{X}))$, so the group of \emph{principal divisors} $\Im L_{\XX}$ lies in $\ZZ^{V(X)}_0$. Hence we can define the Jacobian group of $\XX$ in the same way as for graphs:

\begin{definition} The \emph{Jacobian} group of a graph of groups $\XX$ is the quotient group
\[
\Jac(\XX)=\ZZ^{V(X)}_0/\Im L_{\XX}.
\]
\end{definition}

We give an explicit formula for the matrix $L$ of the Laplacian $L_{\XX}$ of a graph of groups $\XX=(X,\calX_v,\calX_h)$. Assume that $X$ has no legs (this does not affect the Laplacian), and let $n=|V(X)|$ and $m=|E(X)|$ be the number of vertices and edges, respectively. Then $L=Q-A$, where $Q$ is the diagonal \emph{valency matrix} and $A$ is the \emph{adjacency matrix} of the graph of groups $\XX$:
\begin{equation}
L_{uv}=Q_{uv}-A_{uv},\quad
Q_{uv}=\delta_{uv}\sum_{h\in T_vX}\frac{c(v)}{c(h)},\quad A_{uv}=\sum_{h\in T_vX:\, r_X(\iota_X(h))=u}\frac{c(v)}{c(h)}.
\label{eq:QA}
\end{equation}
We note that $L$ and $A$ are not symmetric in general. The Laplacian $L$ is degenerate, specifically its rows sum to zero (but generally not the columns). 

We introduce the following matrix factorizations. Let $C_V$ and $C_E$ be the respectively $n\times n$ and $m\times m$ diagonal matrices 
\[
(C_V)_{uv}=c(u)\delta_{uv},\quad (C_E)_{ef}=c(e)\delta_{ef}
\]
recording the orders of the local groups. Let $S$ and $T$ be the root matrices~\eqref{eq:adjacency} of $X$, with respect to a choice of orientation. It is then elementary to verify that
\begin{equation}
Q=SC_E^{-1}S^tC_V+TC_E^{-1}T^tC_V,\quad A=SC_E^{-1}T^tC_V+TC_E^{-1}S^tC_V,\quad L=(S-T)C_E^{-1}(S-T)^tC_V.
\label{eq:factorization}
\end{equation}

For future use, we also require the adjugate of the Laplacian.

\begin{lemma} The adjugate of the Laplacian matrix $L$ of a graph of groups $\XX=(X,\calX_v,\calX_h)$ is equal to
\[
\adj(L)=C_V^{-1}J\xi.
\]
Here $J$ is the matrix whose entries are all equal to $1$, and the constant $\xi$ is equal to
\[
\xi=\prod_{v\in V(X)}c(v)\sum_{T\subset X}\prod_{e\in E(T)}c(e)^{-1},
\]
where the sum is taken over all spanning trees $T$ of $X$.
\label{lem:adj}
\end{lemma}

\begin{proof} The adjugate of the Laplacian $L$ of an ordinary graph $X$ is equal to $J\cdot \kappa(X)$, where $\kappa(X)=|\Jac(X)|$ is the number of spanning trees, and is computed by applying the Cauchy--Binet formula to the factorization $L=(S-T)(S-T)^t$ (see, for example, Theorem 6.3 in~\cite{1993Biggs}). Applying the same proof to the Laplacian of a graph of groups and using the factorization in Equation~\eqref{eq:factorization} gives the desired result. 

\end{proof}



\begin{remark} We note that defining chip-firing on a graph of groups $\XX=(X,X_v,X_h)$ uses only the underlying graph and the orders $c(v)=|\calX_v|$ and $c(h)=|\calX_h|$ of the local groups. The structure of the groups is irrelevant, which is not surprising given that chip-firing is an abelian theory. In particular, given a group action of $G$ on $X$, the choices of the local stabilizers that are made when defining the quotient graph of groups $X/\!/G$ do not affect chip-firing.

Furthermore, this definition of chip-firing makes sense for any graph whose vertices and edges are equipped with weights $c(v)$ and $c(e)$, with the condition that the weight of any edge divides the weights of its root vertices. The weights themselves need not be integers, so for example rescaling all weights by an arbitrary factor does not change the chip-firing map. This framework allows one to modify the edges and edge weights of a graph without changing the chip-firing map. For example, one may eliminate edge weights entirely by dividing all weights by a sufficiently large number such that each edge $e$ has weight $1/n(e)$ for some integer $n(e)$, and then replacing each edge $e$ with $n(e)$ unweighted edges. Conversely, a set $\{e_1,\ldots,e_n\}$ of edges joining two vertices can be replaced by a single edge $e$ with weight $c(e)=(c(e_1)^{-1}+\cdots+c(e_n)^{-1})^{-1}$, so chip-firing on any weighted graph is equivalent to chip-firing on a simple graph (without multi-edges). Vertex weights, however, cannot be modified away.

\end{remark}

\subsection{The order of the Jacobian via spanning trees} We now compute the order of the Jacobian $\Jac(\XX)=(X,\calX_v,\calX_h)$ of a graph of groups $\XX$ in two different ways. The first formula generalizes Kirchhoff's theorem and computes $\Jac(\XX)$ as a weighted sum over the spanning trees of $X$. A similar formula for a graph with trivial vertex weights appears in Theorem 4.1 in~\cite{2015Chiodo}.

\begin{theorem} \label{thm:matrixtree} Let $\XX=(X,\calX_v,\calX_h)$ be a graph of groups. For each vertex $v\in V(X)$ and edge $e=\{h,h'\}\in E(X)$, let $c(v)=|\calX_v|$ and $c(e)=|\calX_h|=|\calX_{h'}|$ be the orders of the local groups. The order of the Jacobian of $\XX$ is equal to 
\[
\big|\Jac(\XX)\big|=c_v^{-1}\prod_{v\in V(X)}c(v)\sum_{T\subset X}\prod_{e\in E(T)}c(e)^{-1},
\]
where $c_v$ is the least common multiple of the vertex weights $c(v)$, and the sum is taken over all spanning trees $T$ of $X$.
\end{theorem}

\begin{proof} Denote $n=|V(X)|$ and $m=|V(E)|$ and label the vertices of $X$ as $V(X)=\{v_1,\ldots,v_n\}$. Fix an orientation on $X$, then the matrix $L$ of the Laplacian of $\XX$ admits the factorization~\eqref{eq:factorization}
\[
L=BC_E^{-1}B^TC_V,\quad B=S-T,
\]
where $C_E$ and $C_V$ are diagonal matrices recording the $c(e)$ and the $c(v)$. Let $L=[u_1\cdots u_n]$ denote the columns of $L$, these vectors satisfy the relation
\begin{equation}
\frac{u_1}{c(v_1)}+\cdots+\frac{u_n}{c(v_n)}=0.
\label{eq:sumofcolumns}
\end{equation}
The matrix $L$ defines the chip-firing map $L:\ZZ^n\to \ZZ^n$, whose image lies in the kernel of the degree map $\deg:\ZZ^n\to \ZZ$ (which sums the components). Fix the vertex $v_n$ and let $\ZZ^n\to \ZZ^{n-1}$ be the homomorphism that forgets the last coordinate; it is clear that it maps $\Ker \deg$ isomorphically onto $\ZZ^{n-1}$. The matrix of the composed map $\ZZ^n\to \ZZ^n\to \ZZ^{n-1}$ is $L'=[u'_1\cdots u'_n]$, which is $L$ with the last row removed. Then the Jacobian is
\[
\Jac(\XX)=\Ker \deg/\Im L=\ZZ^{n-1}/\Im L'.
\]
Let $\widetilde{L}=[u'_1\cdots u'_{n-1}]$ be the matrix obtained by removing the last column from $L'$, then
\begin{equation}
\big|\Jac(\XX)\big|=\frac{\big|\ZZ^{n-1}/\Im \widetilde{L}\big|}{\big|\Im L'/\Im \widetilde{L}\big|}.
\label{eq:doublefrak}
\end{equation}
The group $\Im L'/\Im \widetilde{L}$ is the finite cyclic group generated by the vector $u'_n$ over the lattice $\langle u_1'\cdots u_{n-1}'\rangle$. Clearing denominators in~\eqref{eq:sumofcolumns}, we obtain the minimal relation between the $u'_i$:
\[
\frac{c_v}{c(v_1)}u'_1+\cdots+\frac{c_v}{c(v_n)}u'_n=0.
\]
Hence the order of $u'_n$ and thus the denominator in~\eqref{eq:doublefrak} is equal to
\begin{equation}
\big|\Im L'/\Im \widetilde{L}\big|=\left|\frac{\langle u_1'\cdots u_n'\rangle}{\langle u_1'\cdots u_{n-1}'\rangle}\right|=\frac{c_v}{c(v_n)}.
\label{eq:extrafactor}
\end{equation}
The numerator in~\eqref{eq:doublefrak} is the determinant of the $(n-1)\times (n-1)$ matrix $\widetilde{L}$ obtained from $L$ by deleting the last row and column. By Lemma~\ref{lem:adj}, it is equal to 
\[
\big|\ZZ^{n-1}/\Im\widetilde{L}\big|=\det \widetilde{L}=\frac{1}{c(v_n)}\xi=\prod_{i=1}^{n-1}c(v_i)\sum_{T\subset X}\prod_{e\in E(T)}c(e)^{-1}.
\]
Plugging the above two equations into~\eqref{eq:doublefrak}, we obtain the result.



\end{proof}

\subsection{The order of the Jacobian via the zeta function} We give an alternative method for computing the order of the Jacobian group $\Jac(\XX)$ of a graph of groups. Recall that the Ihara zeta function $\zeta(u,X)$ of a graph $X$ is an analogue of the Dedekind zeta function of a number field. It is defined as an Euler product over the primes of $X$, which are equivalence classes of certain closed walks on $X$. Unlike its arithmetic analogue, the Ihara zeta function $\zeta(u,X)$ is the reciprocal of an explicit polynomial associated to $X$. Specifically, let $n=|V(X)|$ and $m=|E(X)|$, and let $Q$ and $A$ be the $n\times n$ valency and adjacenty matrices of $X$, then Bass's three-term determinant formula (see~\cite{1992Bass} and~\cite{2010Terras}) states that
\[
\zeta(u,X)^{-1}=(1-u^2)^{m-n}\det(I_n-Au+(Q-I_n)u^2).
\]

The Ihara zeta function of a graph exhibits a number of remarkable similarities to the Dedekind zeta function. For example, it satisfies a graph-theoretic analogue of the class number formula, with $\Jac(X)$ playing the role of the ideal class group. Specifically, at $u=1$ the zeta function has a pole of order $g=m-n+1$ (if $g\geq 2$) and its reciprocal has the following Taylor expansion (see~\cite{1998Northshield}):
\begin{equation}
\label{eq:northshield}
\zeta(u,\XX)^{-1}= 2^g(-1)^{g+1}(g-1)\big |\Jac(\XX)\big|\cdot (u-1)^g+O\left((u-1)^{g+1}\right).
\end{equation}

It is a natural problem to generalize closed walks and the Ihara zeta function to graphs of groups. In~\cite{2021Zakharov}, the second author defined $\zeta(u,\XX)$ for a graph of groups $\XX$ having trivial edge groups and proved an analogue of Bass's three-term determinant formula for $\zeta(u,\XX)$ (see Theorem 3.8 in~\cite{2021Zakharov}), and in upcoming work will extend these results to arbitrary graphs of groups.

It is natural to expect that the Ihara zeta function $\zeta(u,\XX)$ of a graph of groups $\XX$ computes the order of $\Jac(\XX)$. We show that this is indeed the case, provided that $\zeta(u,\XX)$ satisfies an analogue of Bass's three-term determinant formula (which it does in the edge-trivial case by Theorem 3.8 of~\cite{2021Zakharov}).

\begin{theorem} \label{thm:zeta} Let $\XX=(X,\calX_v,\calX_h)$ be a finite graph of groups on a graph with $n=|V(X)|$ vertices and $m=|E(X)|$ edges. Define the Ihara zeta function of $\XX$ by the formula
\[
\zeta(u,\XX)^{-1}=(1-u^2)^{m-n}\det(I_n-Au+(Q-I_n)u^2),
\]
where $Q$ and $A$ are the valency and adjacency matrices~\eqref{eq:QA} of $\XX$. Then $\zeta(u,\XX)^{-1}$ has a zero of order $g=m-n+1$ at $u=1$, and has leading coefficient 
\[
\zeta(u,\XX)^{-1}= 2^g(-1)^{g+1}c_v\left(\sum_{e\in E(X)}c(e)^{-1}-\sum_{v\in V(X)}c(v)^{-1}\right)\big |\Jac(\XX)\big|\cdot (u-1)^g+O\left((u-1)^{g+1}\right),
\]
where $c_v$ is the least common multiple of the vertex weights $c(v)$.

\end{theorem}

\begin{proof} Plugging $u=1$ into the determinant we get
\[
\det(I_n-A+(Q-I_n))=\det L=0,
\]
since the Laplacian is singular. The term $(1-u^2)^{m-n}$ has a zero of order $g-1$ at $u=1$ with leading coefficient $2^{g-1}(-1)^{g+1}$. Therefore $\zeta(u,\XX)$ has a zero of order at least $g$ at $u=1$, and it is sufficient to show that
\[
\left.\frac{d}{du}\det (I_n-Au+(Q-I_n)u^2)\right|_{u=1}=2c_v\big |\Jac(\XX)\big|\left(\sum_{e\in E(X)}c(e)^{-1}-\sum_{v\in V(X)}c(v)^{-1}\right).
\]
We follow the proof of Theorem 2.11 in~\cite{2019HammerMattmanSandsVallieres}. Using Jacobi's formula, we have
\[
\left.\frac{d}{du}\det (I_n-Au+(Q-I_n)u^2)\right|_{u=1}=\left.\tr\left[\adj(I_n-Au+(Q-I_n)u^2)\frac{d}{du}(I_n-Au+(Q-I_n)u^2)\right]\right|_{u=1}=
\]
\begin{equation}
=\tr\left[\adj (Q-A)\cdot(2Q-A-2I_n)\right]=
\tr\adj (L)\cdot Q-2\tr \adj(L),
\label{eq:Jacobi}
\end{equation}
where we used that $L=Q-A$ and therefore
\[
\adj(L)\cdot (Q-A)=\adj(L)\cdot L=\det L\cdot I_n=0.
\]
By Lemma~\ref{lem:adj} and Equation~\eqref{eq:factorization} we have
\[
\tr \adj(L)=\xi\tr (C_V^{-1}J)=\xi \tr(C_V^{-1})=\xi\sum_{v\in V(X)}c(v)^{-1},
\]
\[
\tr \adj(L)\cdot Q=\xi\tr[C_V^{-1}J(SC_E^{-1}S^tC_V+TC_E^{-1}T^tC_V)]=\xi\tr[J(SC_E^{-1}S^t+TC_E^{-1}T^t)]=2\xi\sum_{e\in E(X)}c(e)^{-1},
\]
where
\[
\xi=\prod_{v\in V(X)}c(v)\sum_{T\subset X}\prod_{e\in E(T)}c(e)^{-1}=c_v\big |\Jac(\XX)\big|
\]
by Theorem~\ref{thm:matrixtree}. Plugging these into Equation~\eqref{eq:Jacobi}, we obtain the desired result. 

\end{proof}

\section{The Jacobian of a quotient graph of groups} We now determine the relationship between the Jacobians $\Jac(\tX)$ and $\Jac(\XX)$, where $\tX$ is a graph with a right $G$-action and $\XX=X/\!/G=(X,\calX_v,\calX_h)$ is the quotient graph of groups. 

\subsection{Pushforward and pullback to the quotient}

Let $X=\tX/G$ be the quotient graph, let $p:\tX\to X$ be the quotient map, and let $c(v)=|\calX_v|$ and $c(h)=|\calX_h|$ be the vertex and edge weights. We recall the description of $\tX$ in terms of $\XX$ and a voltage assignment $\beta:H(X)\to G$ given in Section~\ref{subsec:gog}. Following Equation~\eqref{eq:tXset}, we make the identifications
\begin{equation}
\ZZ^{V(\tX)}=\bigoplus_{v\in V(X)}\ZZ^{\calX_v\backslash G},\quad \ZZ^{H(\tX)}=\bigoplus_{h\in H(X)}\ZZ^{\calX_h\backslash G},
\label{eq:tXgroups}
\end{equation}
where the summands correspond to the fibers of $p$. The generators of $\ZZ^{\calX_v\backslash G}$ are denoted $\tv_g$ for $v\in V(X)$ and $g\in G$, where $\tv_g=\tv_{g'}$ if and only if $\calX_vg=\calX_vg'$, and similarly for half-edges. 

It is elementary to verify that, in terms of these identifications, the maps $r_{\tX}$, $\iota_{\tX}$, and $\tau_{\tX}$ are given by the following formulas on the generators:
\begin{equation}
r_{\tX}(\th_g)=\widetilde{r_X(h)}_g,\quad \iota_{\tX}(\th_g)=\widetilde{\iota_X(h)}_{\beta(h)g}, \quad\tau_{\tX}(\tv_g)=\sum_{h\in T_vX}\sum_{g'\in \calX_h\backslash \calX_v}\th_{g'g}.
\label{eq:tXmapsongen}
\end{equation}
We note that the $G$-action on $\tX$ naturally defines right $\ZZ G$-module structures on $\ZZ^{V(\tX)}$ and $\ZZ^{H(\tX)}$, but we do not use this. The various homomorphisms between the free abelian groups associated to the quotient $p:\tX\to X$ are shown on Figure~\ref{fig:5} (the objects in the top row are described in Section~\ref{subsec:kernel}).
\begin{figure}
    \centering
\begin{tikzcd}
\displaystyle\bigoplus_{h\in H(X)} \ZZ^{\calX_h \backslash G}_0\arrow[r,bend left,"r_0"] \arrow[dd,"i_*"'] \arrow[loop left,"\iota_0"]&
\displaystyle\bigoplus_{v\in V(X)} \ZZ^{\calX_v \backslash G}_0 \arrow[l,bend left,"\tau_0"'] \arrow[dd,"i_*"']\arrow[loop right,"L_0"] \\
 & \\
\displaystyle\bigoplus_{h\in H(X)} \ZZ^{\calX_h \backslash G}\arrow[r,bend left,"r_{\tX}"] \arrow[dd,bend left,"p_*"'] \arrow[loop left,"\iota_{\tX}"]&
\displaystyle\bigoplus_{v\in V(X)} \ZZ^{\calX_v\backslash G} \arrow[l,bend left,"\tau_{\tX}"'] \arrow[dd,bend left,"p_*"']\arrow[loop right,"L_{\tX}"] \\
 & \\
\ZZ^{H(X)}\arrow[r,bend left,"r_X"] \arrow[uu,bend left,"p^*"] \arrow[loop left,"\iota_X"]& \ZZ^{V(X)} \arrow[l,bend left,"\tau_{\XX}"'] \arrow[uu,bend left,"p^*"] \arrow[loop right,"L_X"]
\end{tikzcd}
    \caption{Pushforward and pullback maps associated to a quotient}
    \label{fig:5}
\end{figure}
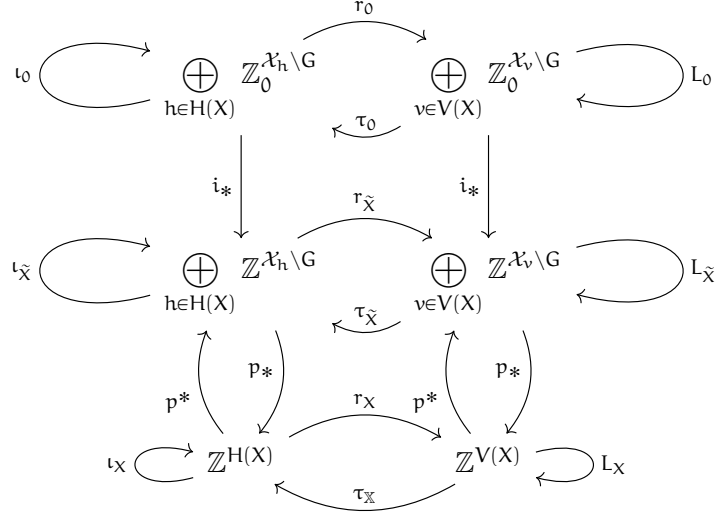

We define the \emph{pushforward} homomorphisms
\[
p_*:\ZZ^{V(\tX)}\to \ZZ^{V(X)},\quad p_*:\ZZ^{H(\tX)}\to \ZZ^{H(X)}
\]
on the generators by the formulas
\[
p_*(\tv_g)=v,\quad v\in V(X),\quad p_*(\th_g)=h,\quad h\in H(X).
\]
We note that the formulas are the same as for a harmonic morphism, in other words, $p_*$ simply adds up the chips in each fiber without any additional weights.

\begin{proposition} The pushforward homomorphism $p_*:\ZZ^{V(\tX)}\to \ZZ^{V(X)}$ commutes with the Laplacians 
\[
p_*\circ L_X=L_{\XX}\circ p_*
\]
and defines a surjective homomorphism $p_*:\Jac(\tX)\to \Jac(\XX)$.
\end{proposition}

\begin{proof} The identities
\[
p_*\circ r_{\tX}=r_X\circ p_*,\quad p_*\circ \iota_{\tX}=\iota_X\circ p_*
\]
hold because $p$ is a morphism of the underlying graphs (though not harmonic in general). It remains to see how $p_*$ interacts with $\tau_{\tX}$ and $\tau_{\XX}$. Let $\tv_g\in V(\tX)$ be a vertex lying over $p(\tv_g)=v$. By Equation~\eqref{eq:tXmapsongen}, we have
\[
(p_*\circ \tau_{\tX})(\tv_g)=p_*\left[\sum_{h\in T_vX}\sum_{g'\in \calX_h\backslash \calX_v}\th_{g'g}\right]=\sum_{h\in T_vX}\sum_{g'\in \calX_h\backslash \calX_v}h=\sum_{h\in T_vX}\frac{|\calX_v|}{|\calX_h|}h,
\]
which is exactly
\[
(\tau_{\XX}\circ p_*)(\tv_g)=\tau_{\XX}(v)=\sum_{h\in T_vX}\frac{c(v)}{c(h)}h.
\]
We therefore see that
\begin{equation}
p_*\circ \tau_{\tX}=\tau_{\XX}\circ p_*,\quad p_*\circ L_X=L_{\XX}\circ p_*,
\label{eq:taulocaldegree2}
\end{equation}
and hence $p_*$ induces a homomorphism $p_*:\Jac(\tX)\to \Jac(\XX)$, which is surjective because the original map $p_*:\ZZ^{V(\tX)}\to \ZZ^{V(X)}$ is surjective.


\end{proof}

We also define a \emph{pullback} homomorphism as follows. Define homomorphisms
\[
p^*:\ZZ^{V(X)}\to \ZZ^{V(\tX)},\quad p^*:\ZZ^{H(X)}\to \ZZ^{H(\tX)}
\]
on the generators as follows:
\begin{equation}
p^*(v)=c(v)\sum_{g\in \calX_v\backslash G}\tv_g, \quad p^*(h)=c(h)\sum_{g\in \calX_h\backslash G}\th_g.
\label{eq:pullbackgog}
\end{equation}

\begin{proposition} The pullback homomorphism $p^*:\ZZ^{V(X)}\to \ZZ^{V(\tX)}$ commutes with the Laplacians
\[
L_{\tX}\circ p^*=p^*\circ L_{\XX}
\]
and defines a homomorphism $p^*:\Jac(\XX)\to \Jac(\tX)$. Furthermore, the homomorphism $p_*\circ p^*$ acts by multiplication by $|G|$ on $\Jac(\XX)$.
\end{proposition}

\begin{proof} Let $h\in H(X)$ be a half-edge rooted at $v=r_X(h)\in V(X)$. Then
\[
(r_{\tX}\circ p^*)(h)=r_{\tX}\left[|\calX_h|\sum_{g\in \calX_h\backslash G}\th_g\right]=|\calX_h|\sum_{g\in \calX_h\backslash G}\tv_g=|\calX_h|\sum_{g\in \calX_v\backslash G}\frac{|\calX_v|}{|\calX_h|}\tv_g=p^*(v)=(p^*\circ r_X)(h),
\]
hence $r_{\tX}\circ p^*=p^*\circ r_X$. Similarly, $\iota_{\tX}\circ p^*=p^*\circ \iota_X$ because $c(\iota_X(h))=c(h)$ for all $h\in H(X)$. Finally, let $v\in V(X)$, then by Equation~\eqref{eq:tauXX} we have
\[
(p^*\circ \tau_{\XX})(v)=p^*\left[\sum_{h\in T_vX}\frac{|\calX_v|}{|\calX_h|}h\right]=\sum_{h\in T_vX}\frac{|\calX_v|}{|\calX_h|}|\calX_h|\sum_{g\in \calX_h\backslash G}\th_g=|\calX_v|\sum_{h\in T_vX}\sum_{g\in \calX_h\backslash G}\th_g,
\]
while by Equation~\eqref{eq:tXmapsongen}
\[
(\tau_{\tX}\circ p^*)(v)=\tau_{\tX}\left[|\calX_v|\sum_{g\in \calX_v\backslash G}\tv_g,\right]=|\calX_v|\sum_{g\in \calX_v\backslash G}\sum_{h\in T_vX}\sum_{g'\in \calX_h\backslash \calX_v}\th_{g'g},
\]
and the two sums agree since each right $\calX_v$-coset is naturally partitioned into $\calX_h$-cosets. Therefore $\tau_{\tX}\circ p^*=p^*\circ \tau_{\XX}$, and putting everything together we get $L_{\tX}\circ p^*=p^*\circ L_{\XX}$. Hence the pullback map induces a homomorphism $p^*:\Jac(\XX)\to \Jac(\tX)$, and $(p_*\circ p^*)(v)=|G|v$ for any $v\in V(X)$ by the orbit-stabilizer theorem.



\end{proof}

We note that, unlike the case of graphs, the pullback homomorphism $p^*$ need not be injective. For example, let $G$ act trivially on any graph $X$, then $\Jac(X/\!/G)=\Jac(X)$ and $p^*:\Jac(X/\!/G)\to \Jac(X)$ acts by multiplication by $|G|$, which is the trivial map if $|G|$ is divisible by $|\Jac(X)|$.

\begin{remark} It is instructive to compare the pushforward $p_*$ and pullback $p^*$ homomorphisms associated to a $G$-cover $p:\tX\to X$ to those associated to a harmonic morphism $f:\tX\to X$. Comparing Equation~\eqref{eq:taulocaldegree1} with~\eqref{eq:taulocaldegree2}, and similarly~\eqref{eq:pullbacklocaldegree} with~\eqref{eq:pullbackgog}, we offer the following stack-theoretic interpretation of the morphisms $p_*$ and $p^*$. The map $p$ views a vertex $\tv\in V(\tX)$ lying over $v=p(\tv)$ as a set of $c(v)$ indistinguishable vertices that have been identified by the $G$-action. The morphism $p$ may then be viewed as a \emph{harmonic morphism} having local degree one at each of these identified vertices. This explains why no degree coefficient appears in Equation~\eqref{eq:taulocaldegree2}, in contrast to Equation~\eqref{eq:taulocaldegree1}. Similarly, the coefficient $c(v)$ in Equation~\eqref{eq:pullbackgog} should be viewed as a count of these identified vertices, and not as a local degree coefficient as in Equation~\eqref{eq:pullbacklocaldegree}. With this interpretation, $p$ is a covering space map (in the stacky sense) of global degree $|G|$. 

\end{remark}

\begin{remark} More generally, one can define the notion of a \emph{harmonic morphism of graphs of groups} $f:\XX\to \YY$ inducing pushforward and pullback homomorphisms $f_*:\Jac(\XX)\to \Jac(\YY)$ and $f^*:\Jac(\YY)\to \Jac(\XX)$. Such a map $f$ is required to satisfy a balancing condition at vertices that takes the local weighs on both $\XX$ and $\YY$ into account. A natural example is the subquotient map $X/\!/H\to X/\!/G$ corresponding a subgroup $H\subset G$ of a group $G$ acting on a graph $X$. We leave the details to the interested reader. \end{remark}

\subsection{Quotients of the tetrahedron} \label{subsec:tetrahedron}

As a simple example, we consider all interesting quotients of $K_4$, the complete graph on $4$ vertices. Denote $V(K_4)=\{a,b,c,d\}$. It is well-known that
\[
\Jac(K_4)\simeq \ZZ/4\ZZ\oplus \ZZ/4\ZZ.
\]
Specifically, $\Jac(K_4)$ is generated by the classes of the divisors
\[
D_a=a-d,\quad D_b=b-d,\quad D_c=c-d
\]
subject to the relations
\[
4D_a=4D_b=4D_c=D_a+D_b+D_c=0.
\]

The automorphism group of $K_4$ is $S_4$, and we consider the quotients $K_4/\!/G$ for all subgroups $G\subset S_4$ that act non-transitively on the vertices (otherwise the quotient graph has a single vertex and its divisor theory is trivial). There are, up to conjugation, four such subgroups, which we enumerate below. The corresponding quotient graphs of groups are shown in Figure~\ref{fig:K4}. Vertices are marked by bold dots, so a line segment with one end vertex represents a leg. Nontrivial stabilizers are labeled by their degree.

\begin{figure}[h]
\centering
\begin{tikzcd}
\begin{tikzpicture}
\draw [ultra thick] (0,0) .. controls (0.5,0.3) and (1,0.3) .. (1.5,0);
\draw [ultra thick] (0,0) .. controls (0.5,-0.3) and (1,-0.3) .. (1.5,0);
\draw[fill](0,0) circle(1mm);
\draw[fill](1.5,0) circle(1mm);
\draw [ultra thick] (0,0) -- (-0.7,0);
\draw [ultra thick] (1.5,0) -- (2.2,0);
\end{tikzpicture}
&&
\begin{tikzpicture}
\draw[fill](0,0) circle(1mm);
\draw[fill](2.0,0.5) circle(1mm);
\draw[fill](1.5,-0.3) circle(1mm);
\draw[fill](0.9,1.5) circle(1mm);
\draw [ultra thick] (0,0) -- (2.0,0.5);
\draw [ultra thick] (0,0) -- (1.5,-0.3);
\draw [ultra thick] (0,0) -- (0.9,1.5);
\draw [ultra thick] (2.0,0.5) -- (1.5,-0.3);
\draw [ultra thick] (2.0,0.5) -- (0.9,1.5);
\draw [ultra thick] (1.5,-0.3) -- (0.9,1.5);
\end{tikzpicture}
\arrow{ll}[swap]{C_{2,2}} \arrow{lld}[swap]{C_3}  \arrow{rrd}{C_2} \arrow{rr}{V_4}
&& 
\begin{tikzpicture}
\draw [ultra thick] (0,0) -- (1.5,0);
\draw[fill](0,0) circle(1mm) node[above]{$2$};
\draw[fill](1.5,0) circle(1mm) node[above]{$2$};
\draw(1.85,0) node[above]{$2$};
\draw(-0.35,0) node[above]{$2$};
\draw [ultra thick] (0,0) -- (-0.7,0);
\draw [ultra thick] (1.5,0) -- (2.2,0);
\end{tikzpicture}
\\
\begin{tikzpicture}
\draw [ultra thick] (0,0) -- (-1.5,0);
\draw [ultra thick] (0,0) .. controls (0.5,0.3) and (0.6,0.1) .. (0.6,0);
\draw [ultra thick] (0,0) .. controls (0.5,-0.3) and (0.6,-0.1) .. (0.6,0);
\draw[fill](0,0) circle(1mm);
\draw[fill](-1.5,0) circle(1mm) node[left]{$3$};
\end{tikzpicture} 
&& 
\,
&& 
\begin{tikzpicture}
\draw [ultra thick] (0,0) -- (-0.7,0);
\draw [ultra thick] (0,0) -- (1.3,0.75) -- (1.3,-0.75) -- (0,0);
\draw[fill](0,0) circle(1mm);
\draw[fill](1.3,0.75) circle(1mm) node[right]{$2$};
\draw[fill](1.3,-0.75) circle(1mm) node[right]{$2$};
\draw(1.3,0) node[right]{$2$};
\end{tikzpicture}
\end{tikzcd}
\caption{Quotients of $K_4$ by non-vertex-transitive group actions.}
\label{fig:K4}
\end{figure}
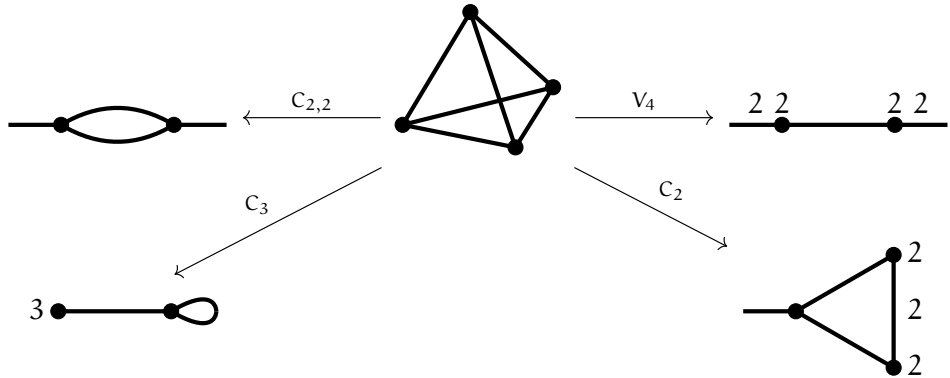

\begin{enumerate}
    \item $C_2$, the order 2 subgroup generated by $(ab)$. The valency, adjacency, and Laplacian matrices of $K_4/\!/C_2$ are
\[
Q=\left(\begin{array}{ccc} 3 & 0 & 0 \\ 0 & 3 & 0 \\ 0 & 0 & 3
\end{array}\right),\quad
A=\left(\begin{array}{ccc} 1 & 2 & 2 \\ 1 & 0 & 1 \\ 1 & 1 & 0
\end{array}\right),\quad
L=\left(\begin{array}{ccc} 2 & -2 & -2 \\ -1 & 3 & -1 \\ -1 & -1 & 3
\end{array}\right).
\]
Finding the Smith normal form of $L$, we see that $\Jac(K_4/\!/C_2)\simeq \ZZ/4\ZZ$. In fact, the Jacobian is generated by the class of $D=p_*(D_a)=p_*(D_b)$, and the pullback map is given by $p^*(D)=D_a+D_b$.

    \item $C_{2,2}$, the order 2 subgroup generated by $(ab)(cd)$. The valency, adjacency, and Laplacian matrices of $K_4/\!/C_{2,2}$ are
\[
Q=\left(\begin{array}{cc} 3 & 0 \\ 0 & 3 
\end{array}\right),\quad
A=\left(\begin{array}{ccc} 1 & 2 \\ 2 & 1 
\end{array}\right),\quad
L=\left(\begin{array}{ccc} 2 & -2 \\ -2 & 2 
\end{array}\right).
\]
The Jacobian is $\Jac(K_4/\!/C_{2,2})\simeq \ZZ/2\ZZ$, generated by $D=p_*(D_a)=p_*(D_b)$, while $p_*(D_c)=0$. The pullback map is $p^*(D)=D_a+D_b-D_c$.

    \item $V_4$, the non-normal Klein 4-group generated by $(ab)$ and $(cd)$. The valency, adjacency, and Laplacian matrices of $K_4/\!/V_4$ are in fact identical to those of $K_4/\!/C_{2,2}$, and the Jacobian is also $\Jac(K_4/\!/V_4)\simeq \ZZ/2\ZZ$.

    \item $C_3$, the order 3 subgroup generated by $(abc)$. The valency, adjacency, and Laplacian matrices of $K_4/\!/C_3$ are
\[
Q=\left(\begin{array}{cc} 3 & 0 \\ 0 & 3 
\end{array}\right),\quad
A=\left(\begin{array}{ccc} 2 & 3 \\ 1 & 0 
\end{array}\right),\quad
L=\left(\begin{array}{ccc} 1 & -3 \\ -1 & 3 
\end{array}\right).
\]
Finding the Smith normal form of $L$, we see that $\Jac(K_4/\!/C_3)$ is the trivial group.
\end{enumerate}

\subsection{Quotients of the Petersen graph} \label{subsec:petersen} As an extended example, we consider the various quotients of the Petersen graph $P$. We identify the vertices of $P$ with two-element subsets of a five-element set $\{a,b,c,d,e\}$. Two vertices are connected by an edge when the corresponding two-element subsets are disjoint. The Jacobian of the Petersen graph is $\Jac(P)\simeq \ZZ/2\ZZ\oplus (\ZZ/10\ZZ)^3$. We have computed the Jacobian $\Jac(P/\!/H)$ of the quotient graph of groups for all subgroups $H\subset \Aut(P)=S_5$, by finding the Smith normal form of the Laplacian. Figure~\ref{fig:subgroups} lists all subgroups $H$, up to conjugacy, having the property that $\Jac(P/\!/H)$ is nontrivial.

The corresponding quotient graphs of groups are shown on Figure~\ref{fig:Petersen}. Each vertex in a quotient graph is labeled by the first vertex of its preimage, with respect to lexicographic order. Numbers at vertices, edges, and legs indicate the orders of nontrivial stabilizers.

\begin{figure}[h]
    \centering

    \begin{tabular}{c|c|c|c}
    Generators of $H$ & Isomorphism class of $H$ & Order of $H$ & $\Jac(P/\!/H)$ \\
    \hline
    $(ab)$ & $\ZZ/2\ZZ$ & $2$ & $(\ZZ/10\ZZ)^2$ \\
    \hline
    $(abc)$ & $\ZZ/3\ZZ$ & $3$ & $\ZZ/5\ZZ$ \\
    \hline
    $(ab)(cd)$ & $\ZZ/2\ZZ$ & $2$ & $\ZZ/2\ZZ\oplus\ZZ/10\ZZ$ \\
    \hline
    $(ab)(cd),(ac)(bd)$ & $(\ZZ/2\ZZ)^2$ & $4$ & $(\ZZ/2\ZZ)^2$ \\
    \hline
    $(abcd)$ & $\ZZ/4\ZZ$ & $4$ & $\ZZ/2\ZZ$ \\
    \hline
    $(ab),(cd)$ & $(\ZZ/2\ZZ)^2$ & $4$ & $\ZZ/10\ZZ$ \\
    \hline
    $(ab),(abc)$ & $S_3$ & $6$ & $\ZZ/5\ZZ$ \\
    \hline
    $(abcd),(ac)$ & $D_4$ & $8$ & $\ZZ/2\ZZ$ 
    \end{tabular}
    \caption{Subgroups $H$ of $\Aut(P)=S_5$ and Jacobians of corresponding quotient graphs of groups. Subgroups resulting in trivial $\Jac(P/\!/H)$ are not listed.}
    \label{fig:subgroups}
\end{figure}

\begin{figure}[h]
    \centering

\begin{tikzcd}
\begin{tikzpicture}
\draw[thick] (0,2) -- (1.5,0) -- (0,-2) -- (-1.5,0) -- (0,2);
\draw[thick] (0,2) -- (0,-2);
\draw[thick] (1.5,0) -- (2.2,0);
\draw[thick] (-1.5,0) -- (-2.2,0);

\draw[fill,white] (1.5,0) circle(3mm);
\draw[thick] (1.5,0) circle(3mm);
\draw (1.5,-0.1) node{$ad$};

\draw[fill,white] (-1.5,0) circle(3mm);
\draw[thick] (-1.5,0) circle(3mm);
\draw (-1.5,-0.1) node{$ac$};

\draw[fill,white] (0,2) circle(3mm);
\draw[thick] (0,2) circle(3mm);
\draw (0,1.9) node{$ce$};

\draw[fill,white] (0,0.67) circle(3mm);
\draw[thick] (0,0.67) circle(3mm);
\draw (0,0.57) node{$ab$};
\draw (0.45,0.52) node{$2$};

\draw[fill,white] (0,-0.67) circle(3mm);
\draw[thick] (0,-0.67) circle(3mm);
\draw (0,-0.77) node{$cd$};
\draw (0.45,-0.82) node{$2$};

\draw[fill,white] (0,-2) circle(3mm);
\draw[thick] (0,-2) circle(3mm);
\draw (0,-2.1) node{$ae$};

\draw (0.15,-0.15) node{$2$};
\end{tikzpicture}
&&
\begin{tikzpicture}
\draw[thick] (0,0) circle(2cm);
\draw[thick] (0,2) -- (0,0);
\draw[thick] (-1.74,-1) -- (0,0);
\draw[thick] (1.74,-1) -- (0,0);

\draw[fill,white] (0,2) circle(3mm);
\draw[thick] (0,2) circle(3mm);
\draw (0,1.9) node{$ad$};

\draw[fill,white] (-1.74,-1) circle(3mm);
\draw[thick] (-1.74,-1) circle(3mm);
\draw (-1.74,-1.1) node{$ac$};

\draw[fill,white] (1.74,-1) circle(3mm);
\draw[thick] (1.74,-1) circle(3mm);
\draw (1.74,-1.1) node{$ae$};

\draw[fill,white] (0,1) circle(3mm);
\draw[thick] (0,1) circle(3mm);
\draw (0,0.9) node{$ce$};
\draw (0.45,0.85) node{$2$};

\draw[fill,white] (-0.87,-0.5) circle(3mm);
\draw[thick] (-0.87,-0.5) circle(3mm);
\draw (-0.87,-0.6) node{$de$};
\draw (-0.87,-1.15) node{$2$};

\draw[fill,white] (0.87,-0.5) circle(3mm);
\draw[thick] (0.87,-0.5) circle(3mm);
\draw (0.87,-0.6) node{$cd$};
\draw (0.87,-1.15) node{$2$};

\draw[fill,white] (0,0) circle(3mm);
\draw[thick] (0,0) circle(3mm);
\draw (0,-0.1) node{$ab$};
\draw (0.55,-0.2) node{$2$};
\draw (0,-0.65) node{$2$};
\draw (-0.55,-0.2) node{$2$};
\draw (0.15,0.35) node{$2$};

\end{tikzpicture}
&&
\begin{tikzpicture}
\draw[thick] (0,2) -- (0,-2) -- (-1.5,0) -- (0,2);
\draw[thick] (-1.5,0) -- (-2.2,0);

\draw[fill,white] (-1.5,0) circle(3mm);
\draw[thick] (-1.5,0) circle(3mm);
\draw (-1.5,-0.1) node{$ac$};

\draw[fill,white] (0,2) circle(3mm);
\draw[thick] (0,2) circle(3mm);
\draw (0,1.9) node{$ce$};
\draw (0.45,1.85) node{$2$};

\draw[fill,white] (0,0.67) circle(3mm);
\draw[thick] (0,0.67) circle(3mm);
\draw (0,0.57) node{$ab$};
\draw (0.45,0.52) node{$4$};

\draw[fill,white] (0,-0.67) circle(3mm);
\draw[thick] (0,-0.67) circle(3mm);
\draw (0,-0.77) node{$cd$};
\draw (0.45,-0.82) node{$4$};

\draw[fill,white] (0,-2) circle(3mm);
\draw[thick] (0,-2) circle(3mm);
\draw (0,-2.1) node{$ae$};
\draw (0.45,-2.15) node{$2$};

\draw (0.15,-0.15) node{$4$};
\draw (0.15,1.18) node{$2$};
\draw (0.15,-1.48) node{$2$};

\end{tikzpicture}
\\
&&
&&
\\
\begin{tikzpicture}
\draw[thick] (0.5,0.87) -- (-1,0) -- (0.5,-0.87);
\draw[thick] (0.5,0.87) .. controls (0.7,0.5) and (0.7,-0.5) .. (0.5,-0.87);
\draw[thick] (0.5,0.87) .. controls (0.3,0.5) and (0.3,-0.5) .. (0.5,-0.87);
\draw[thick] (-2.5,0) -- (-1,0);

\draw[fill,white] (0.5,0.87) circle(3mm);
\draw[thick] (0.5,0.87) circle(3mm);
\draw (0.5,0.77) node{$ad$};

\draw[fill,white] (0.5,-0.87) circle(3mm);
\draw[thick] (0.5,-0.87) circle(3mm);
\draw (0.5,-0.97) node{$ae$};

\draw[fill,white] (-1,0) circle(3mm);
\draw[thick] (-1,0) circle(3mm);
\draw (-1,-0.1) node{$ac$};

\draw[fill,white] (-2.5,0) circle(3mm);
\draw[thick] (-2.5,0) circle(3mm);
\draw (-2.5,-0.1) node{$de$};
\draw (-2.5,-0.65) node{$3$};

\end{tikzpicture}
&&
\begin{tikzpicture}
\draw[thick] (0,0) circle(2cm);
\draw[thick] (0,1)  -- (-0.59, -0.81) -- (0.95,0.31) -- (-0.95,0.31) -- (0.59,-0.81) --(0,1);
\draw[thick] (0,2) -- (0,1);
\draw[thick] (-1.9,0.62) -- (-0.95,0.31);
\draw[thick] (1.9,0.62) -- (0.95,0.31);
\draw[thick] (-1.18,-1.62) -- (-0.59,-0.81);
\draw[thick] (1.18,-1.62) -- (0.59,-0.81);

\draw[fill,white] (0,2) circle(3mm);
\draw[thick] (0,2) circle(3mm);
\draw (0,1.9) node{$ce$};

\draw[fill,white] (-1.9,0.62) circle(3mm);
\draw[thick] (-1.9,0.62) circle(3mm);
\draw (-1.9,0.52) node{$bd$};

\draw[fill,white] (-1.18,-1.62) circle(3mm);
\draw[thick] (-1.18,-1.62) circle(3mm);
\draw (-1.18,-1.72) node{$ac$};

\draw[fill,white] (1.18,-1.62) circle(3mm);
\draw[thick] (1.18,-1.62) circle(3mm);
\draw (1.18,-1.72) node{$be$};

\draw[fill,white] (1.9,0.62) circle(3mm);
\draw[thick] (1.9,0.62) circle(3mm);
\draw (1.9,0.52) node{$ad$};

\draw[fill,white] (0,1) circle(3mm);
\draw[thick] (0,1) circle(3mm);
\draw (0,0.9) node{$ab$};

\draw[fill,white] (-0.95,0.31) circle(3mm);
\draw[thick] (-0.95,0.31) circle(3mm);
\draw (-0.95,0.21) node{$ae$};

\draw[fill,white] (-0.59,-0.81) circle(3mm);
\draw[thick] (-0.59,-0.81) circle(3mm);
\draw (-0.59,-0.91) node{$de$};

\draw[fill,white] (0.59,-0.81) circle(3mm);
\draw[thick] (0.59,-0.81) circle(3mm);
\draw (0.59,-0.91) node{$cd$};

\draw[fill,white] (0.95,0.31) circle(3mm);
\draw[thick] (0.95,0.31) circle(3mm);
\draw (0.95,0.21) node{$bc$};

\end{tikzpicture}
\arrow{uu}[swap]{(ab)}
\arrow{lluu}[swap]{(ab)(cd)}
\arrow{rruu}{(ab),(cd)}
\arrow{ll}[swap]{(abc)}
\arrow{rr}{(abcd)}
\arrow{lldd}{(ab),(abc)}
\arrow{dd}{(ab)(cd),(ac)(bd)}
\arrow{rrdd}{(abcd),(ac)}
&&
\begin{tikzpicture}
\draw[thick] (0,-2.2) -- (0,0);
\draw[thick] (0,1.5) -- (0,2.2);

\draw[thick] (0,1.5) .. controls (0.2,1.2) and (0.2,0.3) .. (0,0);
\draw[thick] (0,1.5) .. controls (-0.2,1.2) and (-0.2,0.3) .. (0,0);

\draw[fill,white] (0,-1.5) circle(3mm);
\draw[thick] (0,-1.5) circle(3mm);
\draw (0,-1.6) node{$ac$};
\draw (0.45,-1.65) node{$2$};
\draw (0.15,-2.15) node{$2$};

\draw[fill,white] (0,0) circle(3mm);
\draw[thick] (0,0) circle(3mm);
\draw (0,-0.1) node{$ae$};

\draw[fill,white] (0,1.5) circle(3mm);
\draw[thick] (0,1.5) circle(3mm);
\draw (0,1.4) node{$ab$};
\end{tikzpicture}
\\
&&
&&
\\
\begin{tikzpicture}
\draw[thick] (0.5,0.87) -- (-1,0) -- (0.5,-0.87) -- (0.5,0.87);
\draw[thick] (-2.5,0) -- (-1,0);

\draw[fill,white] (0.5,0.87) circle(3mm);
\draw[thick] (0.5,0.87) circle(3mm);
\draw (0.5,0.77) node{$ad$};
\draw (0.95,0.72) node{$2$};

\draw[fill,white] (0.5,-0.87) circle(3mm);
\draw[thick] (0.5,-0.87) circle(3mm);
\draw (0.5,-0.97) node{$ae$};
\draw (0.95,-1.02) node{$2$};

\draw[fill,white] (-1,0) circle(3mm);
\draw[thick] (-1,0) circle(3mm);
\draw (-1,-0.1) node{$ac$};
\draw (-1,-0.65) node{$2$};

\draw[fill,white] (-2.5,0) circle(3mm);
\draw[thick] (-2.5,0) circle(3mm);
\draw (-2.5,-0.1) node{$de$};
\draw (-2.5,-0.65) node{$6$};

\draw (-1.75,-0.35) node{$2$};
\draw (-0.25,0.55) node{$2$};
\draw (-0.25,-0.8) node{$2$};

\end{tikzpicture}
&&
\begin{tikzpicture}
\draw[thick] (0,2.25) -- (0,0);
\draw[thick] (-1.95,-1.125) -- (0,0);
\draw[thick] (1.95,-1.125) -- (0,0);

\draw[fill,white] (0,1.5) circle(3mm);
\draw[thick] (0,1.5) circle(3mm);
\draw (0,1.4) node{$ac$};
\draw (0.45,1.35) node{$2$};
\draw (0.15,1.9) node{$2$};

\draw[fill,white] (-1.3,-0.75) circle(3mm);
\draw[thick] (-1.3,-0.75) circle(3mm);
\draw (-1.3,-0.85) node{$ab$};
\draw (-1.3,-1.4) node{$2$};
\draw (-1.7,-1.35) node{$2$};

\draw[fill,white] (1.3,-0.75) circle(3mm);
\draw[thick] (1.3,-0.75) circle(3mm);
\draw(1.3,-0.85) node{$ad$};
\draw(1.3,-1.4) node{$2$};
\draw(1.7,-1.35) node{$2$};

\draw[fill,white] (0,0) circle(3mm);
\draw[thick] (0,0) circle(3mm);
\draw (0,-0.1) node{$ae$};
\end{tikzpicture}
&&
\begin{tikzpicture}
\draw[thick] (0,-2.2) -- (0,2.2);

\draw[fill,white] (0,-1.5) circle(3mm);
\draw[thick] (0,-1.5) circle(3mm);
\draw (0,-1.6) node{$ac$};
\draw (0.45,-1.65) node{$4$};
\draw (0.15,-2.15) node{$4$};
\draw (0.15,-0.85) node{$2$};

\draw[fill,white] (0,0) circle(3mm);
\draw[thick] (0,0) circle(3mm);
\draw (0,-0.1) node{$ae$};
\draw (0.45,-0.15) node{$2$};

\draw[fill,white] (0,1.5) circle(3mm);
\draw[thick] (0,1.5) circle(3mm);
\draw (0,1.4) node{$ab$};
\draw (0.45,1.35) node{$2$};
\draw (0.15,1.9) node{$2$};
\end{tikzpicture}
\end{tikzcd}

    \caption{Quotients of the Petersen graph having non-trivial Jacobian.}
    \label{fig:Petersen}
\end{figure}
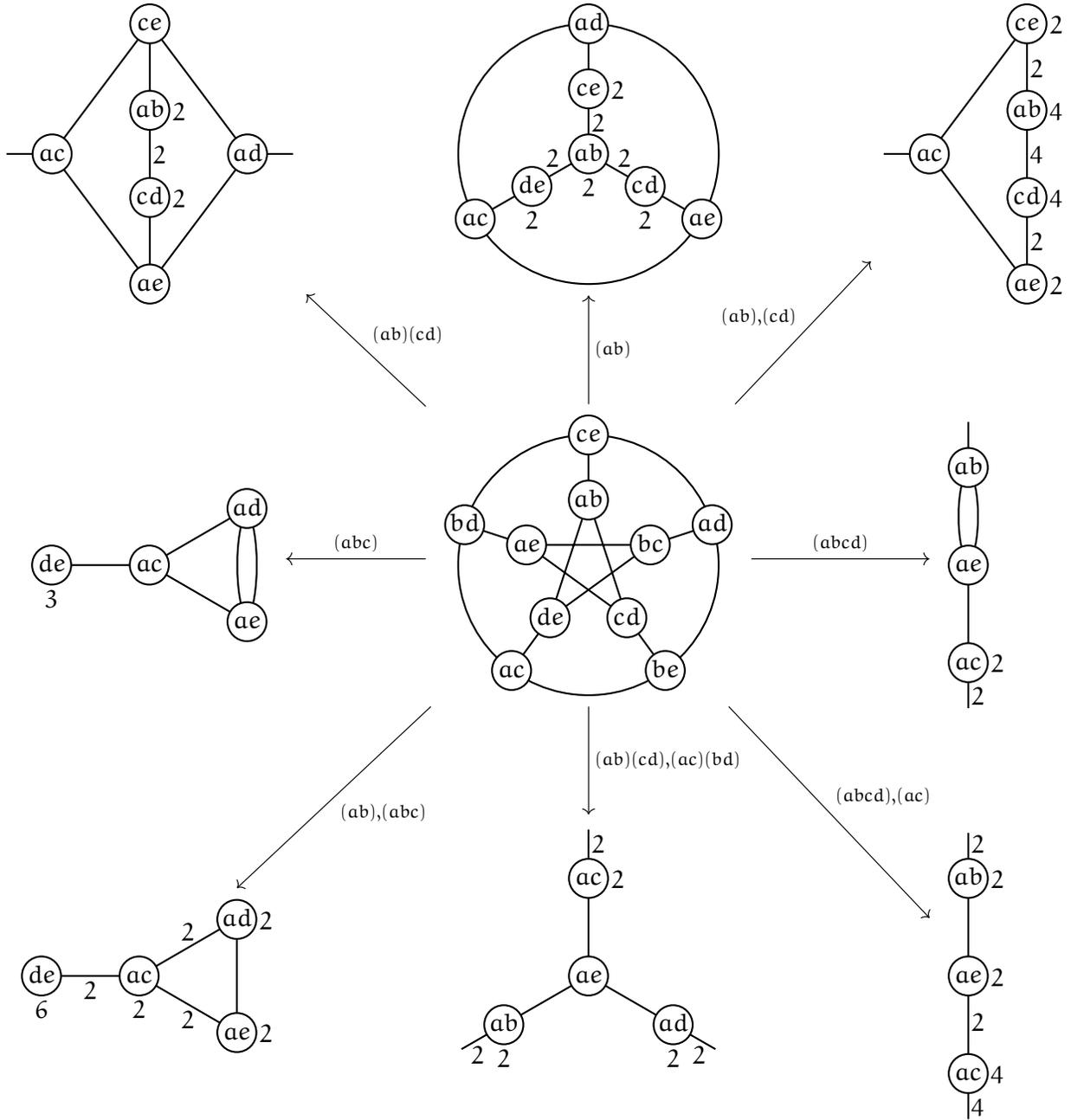

\subsection{The kernel of the pushforward}\label{subsec:kernel} We now identify the kernel of $p_*:\Jac(\tX)\to \Jac(\XX)$, the pushforward map on the Jacobians. Denote the kernels of $p_*$ on $\ZZ^{V(\tX)}$ and $\ZZ^{E(\tX)}$ by 
\begin{equation}
V_0=\Ker\left(p_*:\ZZ^{V(\tX)}\to \ZZ^{V(X)}\right)=\bigoplus_{v\in V(X)}\ZZ^{\calX_v\backslash G}_0,\quad 
H_0=\Ker\left(p_*:\ZZ^{H(\tX)}\to \ZZ^{H(X)}\right)=\bigoplus_{h\in H(X)}\ZZ^{\calX_h\backslash G}_0,
\label{eq:vkhk}
\end{equation}
where we use the identification~\eqref{eq:tXgroups}, and let $i_*:V_0\to \ZZ^{V(\tX)}$ and $i_*:H_0\to \ZZ^{H(\tX)}$ denote the canonical injections. It is elementary to verify that the maps $r_{\tX}$, $\iota_{\tX}$, and $\tau_{\tX}$ descend to maps (see Figure~\ref{fig:5})
\[
r_0:H_0\to V_0,\quad \iota_0:H_0\to H_0,\quad \tau_0:V_0\to H_0.
\]

Following the terminology of~\cite{2014ReinerTseng}, we introduce the following definitions:
\begin{definition} The \emph{voltage Laplacian} of the cover $p:\tX\to X$ is the map 
\[
L_0:V_0\to V_0,\quad L_0=r_0\circ (\Id-\iota_0)\circ \tau_0.
\]
The \emph{voltage Jacobian} of the cover $p:\tX\to X$ is the quotient
\[
\Jac_0=\Im (r_0\circ (\Id-\iota_0))/\Im L_0.
\]
\end{definition}

An elementary rank count shows that the lattices $\Im (r_0\circ (\Id-\iota_0))$ and $\Im L_0$ have full rank in $V_0$. Therefore the voltage Laplacian is non-degenerate, unlike the case of a graph $X$, where $\Im L_X$ has full rank in $\Im (r_X\circ (\Id-\iota_X))=\ZZ^{V(X)}_0$. However, $r_0\circ (\Id-\iota_0)$ is not generally surjective, and the quotients $V_0/\Im L_0$ and $\Jac_0$ need to be carefully distinguished. 

It is clear that $\Jac_0$ embeds into the kernel of $p_*:\Jac(\tX)\to \Jac(\XX)$. In fact, the two are isomorphic.

\begin{proposition} The natural inclusion map $\Jac_0\to \Ker \left(p_*:\Jac(\tX)\to \Jac(\XX)\right)$ is an isomorphism, hence the voltage Jacobian fits into an exact sequence
\begin{equation}
\begin{tikzcd}
    0\ar[r] & \Jac_0\ar[r] & \Jac(\tX)\ar[r,"p_*"] & \Jac(\XX) \ar[r] & 0.
\end{tikzcd}
\end{equation}
In particular, $|\Jac_0|=|\Jac(\tX)|/|\Jac(\XX)|$.
\label{prop:jacobians}
\end{proposition}

\begin{proof} This result generalizes Theorem 1.1 in~\cite{2014ReinerTseng} to the case of non-free $G$-actions, and our proof is essentially a copy of their proof. First, we recall Proposition 2.2 from~\cite{2014ReinerTseng}, which states that, given a diagram $A \overset{f}{\underset{g}\rightleftarrows}B$ of abelian groups, the map $f$ induces an isomorphism
\[
A/(\Im g+\Ker f)\simeq \Im f/\Im (f\circ g).
\]
Hence, denoting
\[
\partial_0=r_0\circ (\Id-\iota_0),\quad
\partial_{\tX}=r_{\tX}\circ (\Id-\iota_{\tX}),\quad
\partial_X=r_X\circ (\Id-\iota_X),
\]
we instead work with the groups
\[
\Jac_0\simeq H_0/(\Im \tau_0+\Ker \partial_0),\quad 
\Jac(\tX)\simeq \ZZ^{H(\tX)}/(\Im \tau_{\tX}+\Ker \partial_{\tX}),\quad 
\Jac(\XX)\simeq \ZZ^{H(X)}/(\Im \tau_{\XX}+\Ker \partial_X).
\]
Second, we replace each of the three finite abelian groups $A=\Jac_0,\Jac(\tX),\Jac(\XX)$ with its Pontryagin dual $A^{\vee}=\Hom(A,\QQ/\ZZ)$. The dual groups are isomorphic, but the arrows now point in the opposite direction:
\[
\begin{tikzcd}
    0 & \Jac_0\ar[l] & \Jac(\tX) \ar[l]& \Jac(\XX) \ar[l,"p^\vee_*"'] & 0\ar[l].
\end{tikzcd}
\]
To show that $\Ker p_*\simeq \Jac_0$, we instead show that $\Coker p_*^{\vee}\simeq \Jac_0$. For each $h\in H(X)$, the map $p_*:\ZZ^{H(\tX)}\to \ZZ^{H(X)}$ sends the generator corresponding to each half-edge $\th\in p^{-1}(h)=\calX_h\backslash G$ to $h$. Hence the Pontryagin dual $p_*:\ZZ^{H(X)}\to \ZZ^{H(\tX)}$ sends $h\in H(X)$ to the sum of the $\th$ over all $\th\in \calX_h\backslash G$. It is therefore clear that $\ZZ^{H(X)}/p_*^{\vee}(\ZZ^{H(X)})\simeq H_0$, and hence 
\[
\Coker p_*^{\vee}=\ZZ^{H(\tX)}/(\Im \tau_{\tX}+\Ker \partial_{\tX}+p_*^{\vee}(\ZZ^{H(X)}))\simeq 
H_0/(\Im \tau_0+\Ker \partial_0)=\Jac_0.
\]

\end{proof}

\begin{remark} Let $p:\tX\to X$ be a free $G$-cover, in other words assume that the $G$-action on $\tX$ is free. By Equation~\eqref{eq:northshield}, the orders of $\Jac(\tX)$ and $\Jac(X)$ can be computed from the Taylor expansions at $u=1$ of the Ihara zeta functions $\zeta(u,\tX)$ and $\zeta(u,X)$. In fact, $\zeta(u,X)$ divides $\zeta(u,\tX)$, and the ratio is a product of the Artin--Ihara $L$-functions $L(u,X,\rho)$ associated to the cover $p:\tX\to X$ corresponding to the nontrivial irreducible representations $\rho$ of $G$ (the $L$-function of the trivial representation is equal to $\zeta(u,X)$, see~\cite{2000StarkTerras} or~\cite{2010Terras}). Hence the order of $\Jac_0$ can likewise be computed by looking at the $u=1$ Taylor expansion of this product. 

Assuming that the Ihara zeta function of a graph of groups is defined and satisfies Bass's three-term determinant formula, Theorem~\ref{thm:zeta} shows that the order $\Jac(\XX)$ can be computed from the Taylor expansion of $\zeta(u,\XX)$ at $u=1$. It is therefore natural to expect that $\zeta(u,\tX)$ is equal to the product of the Artin--Ihara $L$-functions $L(u,\XX,\rho)$ of the graph of groups $\XX$, suitably defined, where the product runs over the irreducible representations of $G$ and where $L(u,\XX,1)=\zeta(u,\XX)$. If this is the case, then $|\Jac_0|=|\Jac(\tX)|/|\Jac(\XX)|$ can be found from the Taylor expansion of the product of the $L$-functions of the cover $\tX\to X$ associated to the nontrivial irreducible representations of $G$. 

The project of defining the Ihara zeta function and the Artin--Ihara $L$-function of a graph of groups was carried out by the second author in~\cite{2021Zakharov} in the case then $G$ acts with trivial stabilizers on the edges of $\tX$. In future work, the second author intends to complete this project and define these functions for arbitrary graphs of groups. 
    
\end{remark}

\section{Double covers} We now consider the group $G=\ZZ/2\ZZ$ acting on a graph $\tX$. We call the quotient map $p:\tX\to X$ a \emph{double cover}, and introduce some terminology borrowed from tropical geometry. 

Let $v\in V(X)$ be a vertex. We say that $v$ is \emph{undilated} if it has two preimages in $\tX$ exchanged by the involution, which we arbitrarily label $p^{-1}(v)=\{\tv^{\pm}\}$, and \emph{dilated} if it has a unique preimage, which we label $p^{-1}(v)=\{\tv\}$. We similarly say that a half-edge $h\in H(X)$ is \emph{undilated} if $p^{-1}(h)=\{\th^{\pm}\}$ and \emph{dilated} if $p^{-1}(h)=\{\th\}$. A dilated half-edge is rooted at a dilated vertex, so the set of dilated half-edges and vertices forms a subgraph $X_{\dil}\subset X$, called the \emph{dilation subgraph}. The root vertex $v=r_X(h)$ of an undilated half-edge $h\in H(X)$ may be dilated or undilated. In the latter case, we label the preimages in such a way that $r_{\tX}(\th^{\pm})=\tv^{\pm}$, in other words a half-edge with a sign is rooted at either a vertex with the same sign or a vertex with no signs. Finally, we say that the double cover $p:\tX\to X$ is \emph{free} if $X_{\dil}=\emptyset$ (in other words, if the $\ZZ/2\ZZ$-action is free) and \emph{dilated} otherwise. 

We now construct the \emph{free graph} $X_{\fr}$ corresponding to the double cover $p:\tX\to X$ as follows. The vertices of $X_{\fr}$ are the undilated vertices of $X$, so $V(X_{\fr})=V(X)\backslash V(X_{\dil})$. The edges of $X_{\fr}$ are the undilated edges of $X$ both of whose root vertices are undilated. The legs of $X_{\fr}$ come in two types. First, each undilated leg of $X$ that is rooted at an undilated vertex is a leg of $X$. Second, consider an edge $e=\{h,h'\}\in E(X)$ having an undilated root vertex $r(h)=u$ and a dilated root vertex $r(h')=v$. For each such edge, we attach $h$ to $X_{\fr}$ as a \emph{leg} rooted at $u$ (so that $r_{X_{\fr}}(h)=r_X(h)=u$ as before but $\iota_{X_{\fr}}(h)=h$ instead of $\iota_X(h)=h'$). We call these \emph{null legs}, in order to distinguish them from the legs coming from $X$. In other words, $X_{\fr}$ is obtained from $X$ by removing $X_{\dil}$, and turning each loose edge (having one root vertex on $X_{\fr}$ and one missing root vertex) into a leg. 

We now define a parity assignment $\ep$ on the half-edges of $X_{\fr}$ as follows: 
\begin{enumerate}

    \item Let $e=\{h_1,h_2\}\in E(X_{\fr})$ be a edge (having undilated root vertices, which may be the same). Our choice of labels for the preimages of the root vertices determines a labeling $\th^{\pm}_1$, $\th^{\pm}_2$ for the preimages of the half-edges. With respect to this choice, we define
\[
\ep(e)=\ep(h_1)=\ep(h_2)=\begin{cases} +1, & \iota_{\tX}(\th_1^{\pm})=\th_2^{\pm},\\ 
-1, & \iota_{\tX}(\th_1^{\pm})=\th_2^{\mp}.
\end{cases}
\]
    We say that $e$ is \emph{even} if $\ep(e)=1$ and \emph{odd} if $\ep(e)=-1$.
    
    \item Let $l\in L(X_{\fr})$ be a leg. If $l$ is a leg of $X$ (in other words, if it is not a null leg), then $p^{-1}(l)=\{\widetilde{l}^{\pm}\}$, and there are two possibilities: either $\iota_{\tX}(\widetilde{l}^{\pm})=\widetilde{l}^{\pm}$, so $p^{-1}(l)$ is a pair of legs exchanged by the involution, or $\iota_{\tX}(\widetilde{l}^{\pm})=\widetilde{l}^{\mp}$, so $e=\{\widetilde{l}^+,\widetilde{l}^-\}$ is an edge folded by the involution. We therefore set
\[
\ep(l)=\begin{cases}
    +1, & \iota_{\tX}(\widetilde{l}^{\pm})=\widetilde{l}^{\pm},\\
    -1, & \iota_{\tX}(\widetilde{l}^{\pm})=\widetilde{l}^{\mp},\\
    0, & l\mbox{ is a null leg.}
\end{cases}
\]
We say that a non-null leg $l$ is \emph{even} if $\ep(l)=1$ and \emph{odd} if $\ep(l)=-1$. 

\end{enumerate} 

The parity assignment $\ep$ gives $X_{\fr}$ the structure of a \emph{signed graph}, and this construction already occurs in~\cite{1982Zaslavsky} for the case of free double covers (so null legs do not appear). The values of $\ep$ on the edges depend the labeling $\tv^{\pm}$ of the preimages $\tv^{\pm}$ of the undilated vertices. The cocycle $[\ep]\in H^1(X_{\fr},\ZZ/2\ZZ)$ in the simplicial cohomology group, however, is well-defined. The leg parity assignement does not depend on any choices, and the cover $p:\tX\to X$ can be uniquely reconstructed from the choice of a dilation subgraph $X_{\dil}\subset X$, an element $[\ep]\in H^1(X_{\fr},\ZZ/2\ZZ)$ defining the edge parity, and a choice of leg parity. 

\subsection{The voltage Laplacian of a double cover} We now compute the voltage Laplacian $L_0$ and the voltage Jacobian $\Jac_0$ of the double cover $p:\tX\to X$ in terms of the free graph $X_{\fr}$. We introduce the following diagram: 
\begin{equation}
\begin{tikzcd}
\ZZ^{H(X_{\fr})}\arrow[r,bend left,"r_{\fr}"] \arrow[loop left,"\iota_{\fr}"]& \ZZ^{V(X_{\fr})} \arrow[l,bend left,"\tau_{\fr}"']. 
\end{tikzcd}
\label{eq:freegraph}
\end{equation}
Here $r_{\fr}=r_{X_{\fr}}$ is the ordinary root map of $X_{\fr}$ and $\tau_{\fr}=\tau_{X_{fr}}$ is its transpose (see Equation~\eqref{eq:tau}). The involution, however, is twisted by the parity assignment:
\begin{equation}
\iota_{\fr}(h)=\ep(h)\iota_{X_{\fr}}(h).
\label{eq:twistediota}
\end{equation}

In terms of the identification given by Equation~\eqref{eq:vkhk}, we have $\ZZ_0^{\calX_v\backslash G}=\ZZ(\tv^+-\tv^-)$ for an undilated vertex $v\in V(X_{\fr})$, while if $v$ is dilated then $\ZZ_0^{\calX_v\backslash G}$ is trivial. Hence we can identify $V_0$ with $\ZZ^{V(X_{\fr})}$. Similarly, $\ZZ_0^{\calX_h\backslash G}=\ZZ(\th^+-\th^-)$ if $h\in H(X)$ is an undilated half-edge and is trivial otherwise. However, $H_0$ is larger than $\ZZ^{H(X_{\fr})}$, since it has generators corresponding to undilated half-edges rooted at dilated vertices. These generators, however, do not appear in the image of $r_0$, and hence we can compute the Laplacian $L_0$ by restricting to $\ZZ^{H(X_{\fr})}$.

\begin{proposition} \label{prop:0isfree} Let $\tX$ be a graph with a $\ZZ/2\ZZ$-action, let $p:\tX\to X$ be the quotient map, let $X_{\fr}$ be the free graph, and let $\ep$ be the parity assignment on $H(X_{\fr})$ defined above. Under the identification of $V_0$ with $V(X_{\fr})$, the voltage Laplacian $L_0:V_0\to V_0$ and the voltage Jacobian are equal to
\[
L_0=r_{\fr}\circ(\Id-\iota_{\fr})\circ \tau_{\fr},\quad \Jac_0=(\Im r_{\fr}\circ(\Id-\iota_{\fr}))/\Im L_0.
\]
The matrix of the voltage Laplacian $L_0:V_0\to V_0$ is explicitly given by 
\[
L_{0,uv}=\begin{cases} |\{\mbox{non-loop edges at }u\}|+4|\{\mbox{odd loops at }u\}|+2|\{\mbox{odd legs at }u\}|+|\mbox{null legs at }u\}|, & u=v,\\
|\{\mbox{odd edges between }u\mbox{ and }v\}|-
|\{\mbox{even edges between }u\mbox{ and }v\}|,& u\neq v.
\end{cases}
\]
    
\end{proposition}

\begin{proof} By abuse of notation, for an undilated vertex $v\in V(X_{\fr})$ we denote $v=\tv^+-\tv^-$ the corresponding generator of $V_0$; this identifies the generators of $\ZZ^{V(X_{\fr})}$ and $V_0$. Similarly, if $h\in H(X)\backslash H(X_{\dil})$ is an undilated edge we denote $h=\th^+-\th^-$ the corresponding generator of $H_0$. If $r_X(h)$ is an undilated vertex then $h$ is also a generator of $\ZZ^{H(X_{\fr})}$, so we view the latter as a subgroup of $H_0$.

It is clear that the maps $\tau_0:\ZZ^{V_0}\to \ZZ^{H_0}$ and $\tau_{\fr}:\ZZ^{V(X_{\fr})}\to \ZZ^{H(X_{\fr})}$ agree under these identifications. Given an undilated half-edge $h\in H(X)\backslash H(X_{\dil})$ rooted at $v=r_X(h)$, we have
\[
r_0(\th^+-\th^-)=\begin{cases} \tv^+-\tv^-,& v\mbox{ is undilated},\\
0, & v\mbox{ is dilated}.
\end{cases}
\]
Hence the restriction of $r_0:\ZZ^{H_0}\to \ZZ^{V_0}$ to $\ZZ^{H(X_{\fr})}$ agrees with $r_{\fr}:\ZZ^{H(X_{\fr})}\to \ZZ^{V(X_{\fr})}$.

Now let $h\in H(X_{\fr})$ be a half-edge rooted at an undilated vertex $v=r_{\fr}(h)$. We need to check that $r_{\fr}\circ(\Id-\iota_{\fr})(h)$ agrees with $r_0\circ(\Id-\iota_0)(\th^+-\th^-)$. There are several cases to consider.

\begin{enumerate}
    \item $h$ is part of an even edge $e=\{h,h'\}\in E(X_{\fr})$, where the vertex $v'=r_{\fr}(h')$ is also undilated. Then $\iota_{\tX}(\th^{\pm})=\th'^{\pm}$, so
    \[
    r_{\fr}\circ(\Id-\iota_{\fr})(h)=r_{\fr}(h-h')=v-v'=\tv^+-\tv^--\tv'^++\tv'^-=r_0\circ(\Id-\iota_0)(\th^+-\th^-).
    \]
    The half-edge $h$ contributes $+1$ to $L_{0,vv}$ and $-1$ to $L_{0,vv'}$, and these contributions cancel if $e$ is a loop. 
    \item $h$ is part of an odd edge $e=\{h,h'\}\in E(X_{\fr})$, where the vertex $v'=r_{\fr}(h')$ is also undilated. Then $\iota_{\tX}(\th^{\pm})=\th'^{\mp}$, so
    \[
    r_{\fr}\circ(\Id-\iota_{\fr})(h)=r_{\fr}(h+h')=v+v'=\tv^+-\tv^-+\tv'^+-\tv'^-=r_0\circ(\Id-\iota_0)(\th^+-\th^-).
    \]
    The half-edge $h$ contributes $+1$ to $L_{0,vv}$ and $+1$ to $L_{0,vv'}$. If $v=v'$ ($e$ is an odd loop), the total contribution from $h$ and $h'$ to $L_{0,vv}$ is equal to $4$.
    \item $h$ is an even leg, then $\iota_{\fr}(h)=h$ and $\iota_{\tX}(\th^{\pm})=\th^{\pm}$ since $\th^{\pm}$ are also legs. Thus
    \[
    r_{\fr}\circ(\Id-\iota_{\fr})(h)=0=r_0\circ(\Id-\iota_0)(\th^+-\th^-)
    \]
    and $h$ does not contribute to the voltage Laplacian.
    \item $h$ is an odd leg and $\th^{\pm}$ form an edge of $\tX$. Then $\iota_{\fr}(h)=-h$ and $\iota_{\tX}(\th^{\pm})=\th^{\mp}$, hence 
    \[
    r_{\fr}\circ(\Id-\iota_{\fr})(h)=2r_{\fr}(h)=2v=2\tv^+-2\tv^-=r_0\circ(\Id-\iota_0)(\th^+-\th^-)
    \]
    and $h$ contributes $+2$ to $L_{0,vv}$.
    \item $h$ is a null leg corresponding to an edge $e=\{h,h'\}\in E(X)$ with dilated root vertex $v'=r_X(h')$. Then $\iota_{\fr}(h)=0$ and we can assume that $\iota_{\tX}(\th^{\pm})=\th'^{\pm}$, so
    \[
    r_{\fr}\circ(\Id-\iota_{\fr})(h)=r_{\fr}(h)=v=\tv^+-\tv^-=r_0\circ(\Id-\iota_0)(\th^+-\th^-)
    \]
    because $r_0(\th'^+-\th'^-)=0$. Hence $h$ contributes $+1$ to $L_{0,vv}$.
    
\end{enumerate}

It follows that $L_0=r_{\fr}\circ(\Id-\iota_{\fr})\circ \tau_{\fr}$, and to complete the proof it is sufficient to show that the image of $H(X_{\fr})\subset H_0$ under the map $r_0\circ (\Id-\iota_0)$ is equal to the image of all of $H_0$. Let $e=\{h,h'\}\in E(X)$ be an undilated edge with undilated root vertex $v=r_X(h)$ and dilated root vertex $v'=r_X(h')$, then $\th'^+-\th'^-$ is a generator of $H_0$ but not $H(X_{\fr})$. We verify that 
\[
r_0\circ (\Id-\iota_0)(\th'^+-\th'^-)=r_0(\th'^+-\th'^--\th^++\th^-)=-\tv^++\tv^-=-v=-r_{\fr}\circ(\Id-\iota_{\fr})(h),
\]
where $h=\th^+-\th^-$ is a generator of $H(X_{\fr})$. Hence adding the $\th'^+-\th'^-$ as a generator to $H(X_{\fr})$ does not increase the image.
\end{proof}

We observe that the matrix of the voltage Laplacian $L_0$ of the double cover $p:\tX\to X$ is obtained from the signed graph Laplacian of the free subgraph $X_{\fr}$ (see Definition 9.4 in~\cite{2014ReinerTseng}) by adding the contributions from the null legs.

\subsection{Ogods and the order of the voltage Jacobian of a double cover} We now derive a combinatorial formula for the order of the voltage Jacobian of a double cover $p:\tX\to X$. To make our formula self-contained, we express it in terms of $\tX$ and $X$, and not in terms of the auxiliary graph $X_{\fr}$. The only terminology that we retain is that we distinguish \emph{odd} and \emph{even} undilated legs of $X$: the preimage of the former is a single edge folded by the involution, while the preimage of the latter is a pair of legs. The following paragraphs are expository, and the interested reader may skip directly to Definition~\ref{def:ogod} and Theorem~\ref{thm:Kirchhoff}.

Kirchhoff's matrix tree theorem states that the order of the Jacobian of a connected graph $X$ is equal to the number of spanning trees of $X$, and a spanning tree of $X$ may be characterized as a minimal connected subgraph containing all vertices of $X$. Our goal is to define an analogous property for subgraphs of the target graph of a double cover.

Let $\tX$ be a graph with a $\ZZ/2\ZZ$-action and let $p:\tX\to X$ be the corresponding double cover. We say that a (possibly disconnected) subgraph $Y\subset X$ is \emph{relatively connected} if each connected component of $Y$ has connected preimage in $\tX$. We now characterize connected subgraphs $Y\subset X$ that are minimal with respect to this property, in other words we require that $p^{-1}(Y)$ be connected but that the graph obtained from $Y$ by removing any edge or leg (and retaining the root vertices) have a connected component with disconnected preimage in $\tX$. We make the following simple observations.

\begin{enumerate}
    \item A connected subgraph $Y\subset X$ having at least one dilated vertex is relatively connected. In particular, $Y$ is not minimally relatively connected if it has at least one dilated edge or leg, since this edge or leg may be removed, or if it has at least two dilated vertices. Similarly, if $Y$ has exactly one dilated vertex but is not a tree, then $Y$ is not minimally relatively connected.
    \item A relatively connected subgraph $Y\subset X$ having at least one even leg is not minimally relatively connected, since the leg may be removed.
    \item A connected subgraph $Y\subset X$ having at least one odd leg $l\in L(Y)$ is relatively connected, since the preimage edge $e=p^{-1}(l)$ connects the (possibly disjoint) preimages of $Y\backslash \{l\}$. The subgraph $Y$ is not minimally relatively connected unless it is a tree.
    
    \item Let $Y\subset X$ be a subgraph containing no dilated vertices and no legs. By covering space theory, the restricted double cover $p|_{p^{-1}(Y)}:p^{-1}(Y)\to Y$ corresponds to an element of $\Hom(\pi_1(Y),\ZZ/2\ZZ)=H^1(Y,\ZZ/2\ZZ)$. If $Y$ is a tree then the cover is trivial and hence disconnected, so $Y$ is not relatively connected. If $Y$ has genus one (in other words, if it has a unique cycle), then $H^1(Y,\ZZ/2\ZZ)=\ZZ/2\ZZ$ and $Y$ has two covers: the trivial disconnected one and the nontrivial connected one. In the latter case, it is clear that $Y$ is minimally relatively connected, since removing any edge produces a tree. Finally, suppose that $Y$ has genus at least two (in other words, it has at least two independent cycles) and $p|_{p^{-1}(Y)}:p^{-1}(Y)\to Y$ is a nontrivial double cover. It is an easy exercise to show that $Y$ is not minimally relatively connected, in other words there is an edge $e\in E(Y)$ such that each connected component of $Y\backslash \{e\}$ (there may be one or two) has connected preimage in $\tX$.
    
\end{enumerate}

We can therefore characterize minimal relatively connected subgraphs of $X$ that contain all vertices of $X$, which are the double cover analogues of spanning trees. One important difference is that these subsets now come with a weight assignment.

\begin{definition} \label{def:ogod}
    
Let $\tX$ be a graph with a $\ZZ/2\ZZ$-action and let $p:\tX\to X$ be the quotient map. An \emph{ogod component} $Y$ of \emph{weight} $w(Y)$ is a connected subgraph $Y\subset X$ having no dilated edges, dilated legs, or even legs, and that is of one of the following three types:
\begin{enumerate}
\item $Y$ is a tree having a unique dilated vertex, and no legs. We say that $w(Y)=1$.
\item $Y$ is a tree having no dilated vertices and a unique odd leg. We say that $w(Y)=2$.
\item $Y$ has no legs and a unique cycle, and $p^{-1}(Y)\subset \tX$ is connected. We say that $w(Y)=4$.
\end{enumerate}
Now let $B$ be a set of $n$ undilated edges and odd legs of $X$, where $n$ is the number of undilated vertices of $X$. Let $X|_B$ be the graph obtained from $X$ by deleting all edges and legs not in $B$, including all dilated edges and legs, and retaining all vertices, and let $X_1,\ldots,X_k$ be the connected components of $X|_B$. We say that $B$ is an \emph{ogod} if each of the $X_i$ is an ogod component, and the \emph{weight} $w(B)$ of the ogod is the product of the weights of the $X_i$.
\end{definition} 

The term \emph{ogod} is an acronym for \emph{odd genus one decomposition}: for a free double cover $p:\tX\to X$ without legs, the connected components $X_i$ of an ogod are graphs of genus one such that the restricted covers $p|_{p^{-1}(X_i)}:p^{-1}(X_i)\to X_i$ are given by the odd (nontrivial) elements of $H^1(X_i,\ZZ/2\ZZ)$. This terminology was introduced by the second author in~\cite{2022LenZakharov}, who was unaware of the history of this definition going back to the seminal paper~\cite{1982Zaslavsky}. Howver, to the best of the authors' knowledge, there does not appear to be an established term describing such subsets in the combinatorics literature.

We are now ready to state the analogue of Kirchhoff's matrix tree theorem for a dilated double cover $p:\tX\to X$, with ogods playing the role of spanning trees. 

\begin{theorem} \label{thm:Kirchhoff} Let $\tX$ be a graph with a non-free $\ZZ/2\ZZ$-action and let $p:\tX\to X$ be the quotient map. The order of the voltage Laplacian is equal to
\begin{equation}
\label{eq:ogodformula}
|\Jac_0|=\sum_{B}w(B),
\end{equation}
where the sum is taken over all ogods $B$ of $X$.

\end{theorem}
For free double covers, this result already occurs in~\cite{1982Zaslavsky}, and was explicitly interpreted as a formula for the order of the voltage Laplacian in~\cite{2014ReinerTseng}. It was subsequently independently derived by the second author in~\cite{2022LenZakharov}. We note that for a free double cover there is an additional $1/2$ coefficient in the right hand side of Equation~\eqref{eq:ogodformula}.

\begin{proof} Let $X_{\fr}$ be the free graph, and let $\ep$ be the parity assignment on $H(X_{\fr})$ defined above. By Proposition~\ref{prop:0isfree}, we may compute the voltage Laplacian $L_0=r_{\fr}\circ(\Id-\iota_{\fr})\circ \tau_{\fr}$ and voltage Jacobian $\Jac_0=(\Im r_{\fr}\circ(\Id-\iota_{\fr}))/\Im L_0$ using the diagram~\eqref{eq:freegraph} of $X_{\fr}$. Let $n=|V(X_{\fr})|$ and $m=|H(X_{\fr})|$. The $n\times n$ matrix of the voltage Laplacian factors as $L_0=DT$, where $D$ is the $n\times m$ matrix of $r_{\fr}\circ (\Id-\iota_{\fr})$ and $T$ is the $m\times n$ matrix of $\tau_{\fr}$:
\[
D_{vh}=\begin{cases} +1, & r(h)=v\mbox{ and }h\mbox{ lies on a non-loop edge or is a null leg},\\
+1, & r(\iota(h))=v\mbox{ and }h\mbox{ lies on an odd non-loop},\\
-1, & r(\iota(h))=v\mbox{ and }h\mbox{ lies on an even non-loop},\\
+2, & r(h)=v\mbox{ and }h\mbox{ lies on an odd loop or is an odd leg},\\
0, & \mbox{otherwise},
\end{cases}
\quad
T_{hv}=\begin{cases} +1, & v=r_{\fr}(h),\\
0, & \mbox{otherwise.}
\end{cases}
\]
By the Cauchy--Binet formula,
\begin{equation}
\det L_0=\sum_{B\subset H(X_{\fr}):|B|=n} \det D|_B\det T|_B,
\label{eq:CB}
\end{equation}
where we sum over all $n$-element subsets $B\subset H(X_{\fr})$ of half-edges of $X_{\fr}$ and where $D|_B$ and $T|_B$ are the matrices obtained from $D$ and $T$ by deleting respectively all columns and all rows except those indexed by $B$. 

We make a number of simple observations:

\begin{enumerate}
    \item $\det D|_B=0$ if $B$ contains a half-edge that lies on an even loop or is an even leg. Indeed, the corresponding column of $D$ is zero.
    \item $\det D|_B=0$ if $B$ contains both half-edges of a single edge $e=\{h,h'\}$. Indeed, the $h$- and $h'$-columns of $D$ are equal if $e$ is odd and sum to zero if $e$ is even. Hence we only consider only those $n$-element subsets $B\subset H(X_{\fr})$ that have at most one half-edge from each edge. We represent each such $B$ as a choice of a total of $n$ edges and legs, as well as an \emph{orientation} for each edge, in other words an arrow pointing in the direction of the chosen half-edge. 
    \item $\det T|_B=0$ unless each half-edge in $B$ is rooted at a distinct vertex of $X_{\fr}$. Viewing $B$ as a choice of oriented edges and legs, we require that each arrow point to a different vertex. 
\end{enumerate}

We now show that the nonzero contributions in Equation~\eqref{eq:CB} come from ogods, and that the contribution from each ogod $B$ is exactly $w(B)$. Fix $B$, and let $X_{\fr}|_B$ be the subgraph of $X$ obtained by deleting all edges and legs not in $B$. Let $X_{\fr}|_B=X_1\cup\cdots\cup X_k$ be the decomposition into connected components, and let $B_i=H(X_i)\cap B$ for $i=1,\ldots,k$. The matrices $D|_B$ and $T|_B$ are block-diagonal with blocks corresponding to the $X_i$, and a block-diagonal matrix has nonzero determinant only if each block is square, in other words if $|B_i|=|V(X_i)|$ for each $i$. In other words, the product $\det D|_B\det T|_B$ is nonzero only if each $X_i$ is a connected oriented graph having an equal number of legs and edges as vertices, with each leg and edge pointing to a distinct vertex. A moment's thought shows that there are only two possibilities for each $X_i$: 

\begin{enumerate}
    \item $X_i$ has a unique leg (odd or null but not even) and is a tree, and all edges are oriented away from the root vertex of the leg. Hence $X_i$ is an ogod component of weight $w(X_i)=1$ if the leg is null and $w(X_i)=2$ if the leg is odd. 
    \item $X_i$ has no legs and a unique cycle. The edges on the cycle are oriented cyclically, while the remaining edges (lying on trees attached to the cycle) are oriented away from the cycle. Hence $X_i$ is an ogod component of weight $w(X_i)=4$ if the preimage of the cycle is connected, which happens if an odd number of edges on the cycle are odd. If there is an even number of odd edges, then the preimage of the cycle is disconnected and $X_i$ is not an ogod.
    
\end{enumerate}

It is now an elementary linear algebra exercise to show that the product $\det D|_{B_i}\det T|_{B_i}$ equals $1$ or $2$ in the first case, depending on whether the unique leg is null or odd. Similarly, in the second case the product is equal to $2$ if there is an odd number of odd edges along the cycle and zero if there is an even number. In this case, there are two contributions corresponding to the two possible choices of orientation along the cycle. Hence we see that the total contribution of $\det D|_{B_i}\det T|_{B_i}$ from an ogod component $X_i$ is equal to $w(X_i)$. Since weights and determinants are multiplicative in connected components, it follows that the contribution of each ogod $B$ to the sum of the $\det D|_{B}\det T|_{B}$ (taken over the possible choices of orientations) is equal to $w(B)$.

We have shown that $\det L_0$ is equal to the right hand side of Equation~\eqref{eq:ogodformula}. To complete the proof, we show that the map $r_{fr}\circ (Id-\iota_{\fr}):\ZZ^{H(X_{fr})}\to \ZZ^{V(X_{\fr})}$ is surjective (this is in contrast to free double covers, where the image has index two). Again, we may pass to connected components and assume that $X_{\fr}$ is connected. Since the double cover $p:\tX\to X$ is dilated, there is at least one dilated vertex $v\in V(X)\backslash V(X_{\fr})$ connected by an undilated edge to an undilated vertex $u\in V(X_{\fr})$. Let $l\in L(X_{\fr})$ be the corresponding null leg rooted at $u$. By the proof of Proposition~\ref{prop:0isfree} we have $r_{fr}\circ (Id-\iota_{\fr})(l)=u$, so $u\in \Im (r_{fr}\circ (Id-\iota_{\fr}))$. Now let $e=\{h,h'\}\in E(X_{\fr})$ be an edge rooted at $r(h)=u$ and another vertex $r(h')=u'$. Again by the proof of Proposition~\ref{prop:0isfree} we have $r_{fr}\circ (Id-\iota_{\fr})(h)=u\pm u'$, and since $u\in (\Im r_{\fr}\circ (Id-\iota_{\fr}))$ we have $u'\in (\Im r_{fr}\circ (Id-\iota_{\fr}))$. Since $X_{\fr}$ is connected, we may proceed in this way and show that $w\in (\Im r_{fr}\circ (Id-\iota_{\fr}))$ for every generator $w$ of $\ZZ^{V(X_{\fr})}$. This completes the proof.

\end{proof}

\begin{example} We consider the two $\ZZ/2\ZZ$-quotients of the Petersen graph $P$ shown on Figure~\ref{fig:Petersen}. We recall that $\Jac(P)=\ZZ/2\ZZ\oplus (\ZZ/10\ZZ)^3$ and thus $|\Jac(P)|=2000$.

Taking the quotient by the order two subgroup $G\subset\Aut(P)$ generated by $(ab)$, we obtain the top center graph $P/G$. There are three undilated vertices $ac$, $ad$, and $ae$ and six undilated edges that we denote $E_u=\{e_{ac,de}, e_{ad,cd}, e_{ae,cd}, e_{ac,ad}, e_{ad,ae},e_{ac,ae}\}$. We consider the $20$ three-element subsets of $E_u$. If we remove the three edges of $P/G$ incident to $ac$, then the lone vertex $ac\in V(P/G)$ has disconnected preimage $p^{-1}(ac)=\{ac,bc\}$. Hence $B=\{e_{ac,de},e_{ac,ad},e_{ac,ae}\}$ is not an ogod, and the same is true for the tangent spaces to $ad$ and $ae$. The outside cycle $B=\{e_{ac,ad}, e_{ad,ae}, e_{ae,ac}\}$ lifts to a closed loop in $P$ and hence is an ogod of weight $4$. For each of the $16$ remaining $3$-element subsets $B\subset E_u$, every connected component of the graph $(P/G)|_B$ is a tree having a unique dilated vertex, hence $B$ is an ogod of weight $1$. Proposition~\ref{prop:jacobians} and Theorem~\ref{thm:Kirchhoff} imply that
\[
\frac{|\Jac(P)|}{|\Jac(P/\!/G)|}=|\Jac_0|=16+1\cdot 4=20.
\]
This agrees with Figure~\ref{fig:subgroups}, since $\Jac(P/\!/G)=(\ZZ/10\ZZ)^2$ and hence $|\Jac(P/\!/G)|=100$.

We also consider the order two subgroup $H\subset \Aut(P)$ generated by $(ab)(cd)$, the quotient graph for which is the top left graph in Figure~\ref{fig:Petersen}. The graph $P/\!/H$ has six undilated edges $E_u=\{e_{ab,ce}, e_{ac,ce}, e_{ac,ae}, e_{ad,ae}, e_{ad,ce},e_{ae,cd}\}$ and two odd legs $L=\{l_{ac},l_{ad}\}$. Out of the 70 4-element subsets of $E_u\cup L$, there are 46 ogods in 15 symmetry classes. Figure~\ref{fig:ogodtable} lists all ogods up to symmetry together with their weights. The total weight of all ogods is $100$, so by Proposition~\ref{prop:jacobians} and Theorem~\ref{thm:Kirchhoff} we have
\[
\frac{|\Jac(P)|}{|\Jac(P/\!/H)|}=|\Jac_0|=100
\]
This agrees with Figure~\ref{fig:subgroups}, since $\Jac(P/\!/G)=\ZZ/2\ZZ\oplus \ZZ/10\ZZ$ and hence $|\Jac(P/\!/G)|=20$.

which agrees with Figure~\ref{fig:subgroups} since $\Jac(P/\!/G)=(\ZZ/10\ZZ)^2$ and hence $|\Jac(P/G)|=100$.

\begin{figure}
    \centering

    \begin{tabular}{|c|c|c|} \hline
    ogod & number of symmetric ogods & weight \\
    \hline
    $\{e_{ac,ae},e_{ac,ce},e_{ad,ae},e_{ad,ce}\}$ & $1$ & $4$ \\
    \hline
    $\{e_{ab,ce},e_{ac,ae},e_{ac,ce},e_{ad,ae}\}$ & $4$ & $1$ \\
    \hline
    $\{e_{ab,ce},e_{ac,ae},e_{ac,ce},e_{ad,ce}\}$ & $4$ & $1$ \\
    \hline
    $\{e_{ab,ce},e_{ac,ae},e_{ad,ae},e_{ae,cd}\}$ & $2$ & $1$ \\
    \hline
    $\{e_{ab,ce},e_{ac,ae},e_{ad,ce},e_{ae,cd}\}$ & $2$ & $1$ \\
    \hline
    $\{l_{ac},e_{ac,ae},e_{ad,ae},e_{ad,ce}\}$ & $4$ & $2$ \\
    \hline
    $\{l_{ac},e_{ac,ae},e_{ac,ce},e_{ad,ae}\}$ & $4$ & $2$ \\
    \hline
    $\{l_{ac},e_{ab,ce},e_{ac,ae},e_{ad,ce}\}$ & $4$ & $2$ \\
    \hline
    $\{l_{ac},e_{ab,ce},e_{ad,ae},e_{ad,ce}\}$ & $4$ & $2$ \\
    \hline    
    $\{l_{ac},e_{ab,ce},e_{ac,ae},e_{ad,ae}\}$ & $4$ & $2$ \\
    \hline    
    $\{l_{ac},e_{ab,ce},e_{ad,ae},e_{ae,cd}\}$ & $4$ & $2$ \\
    \hline
    $\{l_{ac},l_{ad},e_{ac,ae},e_{ac,ce}\}$ & $2$ & $4$ \\
    \hline   
    $\{l_{ac},l_{ad},e_{ab,cd},e_{ac,ae}\}$ & $4$ & $4$ \\
    \hline  
    $\{l_{ac},l_{ad},e_{ac,ae},e_{ad,ce}\}$ & $2$ & $4$ \\
    \hline
    $\{l_{ac},l_{ad},e_{ab,ce},e_{ae,cd}\}$ & $1$ & $4$ \\
    \hline
    \end{tabular}
    \caption{Ogods of the quotient $P\to P/H$ of the Petersen graph for $H=\{1,(ab)(cd)\}$.}
    \label{fig:ogodtable}
\end{figure}

\end{example}

\bibliographystyle{amsalpha}
\bibliography{biblio}{}

\end{document}